\documentclass[journal]{IEEEtran}
\usepackage[utf8]{inputenc}
\usepackage{amsmath,amsthm}
\usepackage{mathtools}
\mathtoolsset{showonlyrefs}

\usepackage{tikz}
\usetikzlibrary{decorations.markings, positioning, intersections}

\usepackage{amssymb}

\usepackage{enumitem}
\setdescription{leftmargin=1em, style=sameline}

\usepackage{xparse}

\usepackage[bookmarksnumbered,bookmarksopen,
unicode]{hyperref}
\setlist[enumerate,1]{label=(\roman*)}
\setlist[enumerate,2]{label=(\arabic*),
  ref=(\roman{enumi}.\arabic*)}
\setlist[enumerate,3]{label=(\alph*),
  ref=(\roman{enumi}.\arabic{enumii}.\alph*)}

\newcommand{\interior}[1]{\operatorname{int}\left(#1\right)}
\newcommand{\card}[1]{\left|#1\right|}
\DeclarePairedDelimiter\abs{\lvert}{\rvert}
\newcommand{\tnorm}[1]{\norm[2]{#1}}
\newcommand{\binSet}{\{0,1\}}
\newcommand{\sprod}[2]{\left \langle #1,#2 \right \rangle}
\newcommand{\N}{\mathbb N}

\newcommand*{\set}[1]{\left\{#1\right\}}

\NewDocumentCommand{\probability}{d()om}{\operatorname{\mathbb{P}}%
  \IfValueT{#1}{\sb{#1}}%
  \left[#3\IfValueT{#2}{\,\middle|\,#2}\right]}
\NewDocumentCommand{\expectation}{d()om}{\operatorname{\mathbb{E}}%
  \IfValueT{#1}{\sb{#1}}%
  \left[#3\IfValueT{#2}{\,\middle|\,#2}\right]}
\NewDocumentCommand{\entropy}{om}{\operatorname{\mathbb{H}}\left[#2%
    \IfValueT{#1}{\,\middle|\,#1}\right]}
\NewDocumentCommand{\mutualInfo}{omm}{\operatorname{\mathbb{I}}%
  \left[#2;#3%
    \IfValueT{#1}{\,\middle|\,#1}\right]}
\newcommand*{\bentropy}[1]{\operatorname{\mathbb{H}}\left[#1\right]}

\newcommand*{\size}[1]{\left|#1\right|}
\newcommand*{\norm}[2][]{\left\|#2\right\|\sb{#1}}

\newcommand*{\prefix}[2]{#2|_{#1}}
\newcommand*{\flipCut}[2]{#2^{\oplus(#1)}}
\DeclareMathOperator{\range}{range}

\DeclareMathOperator{\Compl}{Compl}
\DeclareMathOperator{\dom}{dom}

\tikzset{
  namey/.style=
  {insert path=
    {edge[dashed, thin] node[left, at end] {#1}
      (\tikztostart -| 0,0)}},
  namex/.style=
  {insert path=
    {edge[dashed, thin] node[below, at end] {#1}
      (\tikztostart |- 0,0)}},
  label/.style={%
    postaction={ decorate,
      decoration={ markings, mark=at position .5 with
        \node[above left]{#1};}}}}

\newtheoremstyle{case}{}{}{}{}{\scshape}{:}{ }{}
\theoremstyle{case}
\newtheorem{case-internal}{Case}
\newenvironment{case}[1]{\begin{case-internal}\emph{#1}}
  {\end{case-internal}}
\newenvironment{casesblock}{\setcounter{case-internal}{0}}{}

\newcommand{\RR}{\mathbb{R}}
\DeclareMathOperator{\sign}{sgn}

\theoremstyle{plain}
\newtheorem{theorem}{Theorem}[section]
\newtheorem{lemma}[theorem]{Lemma}
\newtheorem{proposition}[theorem]{Proposition}

\theoremstyle{definition}
\newtheorem{orac}[theorem]{Oracle}
\newtheorem{definition}[theorem]{Definition}
\newtheorem{fact}[theorem]{Fact}

\theoremstyle{remark}
\newtheorem{remark}[theorem]{Remark}

\title{Lower Bounds on the Oracle Complexity of Nonsmooth Convex Optimization via\\ Information Theory}
\author{
\IEEEauthorblockN{G\'abor Braun, Crist\'obal 
Guzm\'an, Sebastian Pokutta}
\thanks{G. Braun is with the Department of Industrial \& Systems 
Engineering at Georgia Institute of Technology. C. Guzm\'an is with
Facultad de Matem\'aticas \& Escuela de Ingenier\'ia at 
Pontificia Universidad Cat\'olica de Chile. S. Pokutta is with with the Department of Industrial \& Systems 
Engineering at Georgia Institute of Technology.}
\thanks{The authors would like to thank Arkadi Nemirovski and Fran\c{c}ois Glineur for the valuable
discussions.} 
\thanks{Research reported in this paper was partially supported
by NSF grants  CMMI-1232623 and CMMI-1300144.}
\thanks{Manuscript submitted July 19, 2014; revise April 19, 2017.}
}

\date{}
\begin{document}
\maketitle
%\IEEEpeerreviewmaketitle
\begin{abstract}
  We present an information-theoretic approach to lower bound the 
  oracle complexity of nonsmooth black box convex optimization, 
  unifying previous lower bounding techniques 
  %in the spirit of \cite{p2011unifying}
  by identifying a combinatorial problem,
  namely string guessing, as a single source of hardness.
  As a measure of complexity we use
  distributional oracle complexity,
  which subsumes randomized oracle complexity
  as well as worst-case oracle complexity.
  We obtain strong lower bounds
  on distributional oracle complexity
  for the box \([-1,1]^{n}\),
  as well as for the \(L^p\)-ball for \(p \geq 1\)
  (for both low-scale and large-scale regimes),
  matching worst-case upper %lower
  bounds,
  and hence we close the gap
  between distributional complexity, and in particular, 
  randomized complexity, and worst-case complexity.
  Furthermore, the bounds remain essentially the 
  same for high-probability and bounded-error
  oracle complexity, and even for combination of the two,
  i.e., bounded-error high-probability oracle
  complexity.
  This considerably
  extends the applicability of known bounds.
\end{abstract}
\vspace{0.2cm}


\section{Introduction}
\label{sec:introduction}

For studying complexity of algorithms,
oracle models are popular
to abstract away from computational
resources,
i.e., to focus on \emph{information}
instead of \emph{computation} as the limiting resource.
Therefore oracle models typically measure complexity by
the number of required queries to the oracle,
whose bounds are often not subject to
strong computational complexity assumptions, such as \(P \neq NP\).

We study the complexity of nonsmooth convex optimization in the
standard black box model.
The task is to find the optimum of a function \(f\),
which is only accessible through a local oracle.
The oracle can be queried with any point \(x\) of the domain,
and provides information about \(f\)
in a small neighborhood of \(x\).
This generic model captures the behavior of most 
first-order methods, successfully applied
in engineering \cite{BenTal:2001},
machine learning \cite{Jaggi:2013},
image and signal processing
\cite{TV_Beck:2009,Wright:2010},
and compressed sensing
\cite{ell1_Beck:2009,Nemirovski:2013}.
All these applications require only
medium accuracy solutions
even with noisy data.
Moreover, in the era of big data,
other general-purpose methods,
such as interior-point or Newton-type,
are prohibitively expensive.
This makes a strong case for using first-order methods,
and  the black box model has been extensively 
studied for various function classes
and domains such as \(L^p\)-balls and the box.
Most of the lower bounds were established in
\cite{Nemirovski:1983,Nemirovski:1994,Nesterov:2004}.
These bound worst-case complexity
using the technique of \emph{resisting oracles},
continuously changing the function \(f\) to provide the less
informative consistent answers.
In particular, for Lipschitz-continuous convex functions
on the \(n\)-dimensional \(L^p\)-ball for \(1 \leq p < \infty\),
depending on an additive accuracy \(\varepsilon\),
two regimes of interest were established:
the high accuracy or \emph{low-scale} regime,
where the dimension \(n\) appears as a multiplicative term
in the complexity $\Theta(n\log 1/\varepsilon)$,
and the low accuracy or \emph{large-scale} regime,
where the complexity $\Theta(1/\varepsilon^{\max\{p,2\}})$
is independent of the dimension.
Interestingly,
each of the two regimes has its own optimal method:
the Center of Gravity method in the low-scale regime,
and the Mirror-Descent method in the large-scale
\cite{Nemirovski:1983}.

We provide a unification of lower bounds on the oracle 
complexity for nonsmooth convex optimization.
We will also identify a core combinatorial problem,
namely, a string guessing problem,
from which we derive all our lower bounds for convex optimization.
Thus, we obtain strong lower bounds
on distributional oracle complexity in the nonsmooth case,
matching all known bounds of the
worst-case. In fact, we will even show that these bounds do not only
hold in expectation but also with high
probability, and even for Monte Carlo algorithms,
which provide correct answer only with a bounded error probability.

The core problem will be handled by information theory,
which is a natural approach due to
the informational nature of oracle models.
Information theory has been prominently used
to obtain strong lower bounds in other complexity problems as well.

\subsection*{Related work}

Our approach through information theory
was motivated by the following works obtaining
information-theoretic lower bounds on:
communication
\cite{braverman2011information,braverman2012information3},
data
structures \cite{p2011unifying,dasgupta2012sparse},
extended formulations
\cite{braverman2012information,BP2013commInfo,BFP2013},
streaming computation \cite{chakrabarti2013information}, and many
more. 
Lower bounds were established
for many other classical oracle settings,
such as submodular function minimization with
access to function value oracles
\cite{goel2009approximability,iwata2009submodular,%
  svitkina2011submodular}, however
typically not explicitly relying on
information theory but rather bounding the randomized complexity by
means of Yao's minimax principle.
For first-order oracles,
algorithms have been proposed \cite{chudak2007efficient},
however next to nothing is
known about strong lower bounds. 

As pointed out to us by an anonymous reviewer, 
the string guessing problem has already been used
for lower bounds on the advice
complexity for online combinatorial algorithms 
\cite{Bockenhauer:2013}. In this work, we use string
guessing as a base problem for deriving oracle lower bounds.
However, the version of string guessing in \cite{Bockenhauer:2013} 
is slightly different from ours:
There one bit must be predicted at a time---with
or without advice---and the cost is the total number
of wrongly guessed bits.
The difference becomes essential for erroneous oracles.
Further study of the connections between this version of the problem and, e.g.,
online convex optimization, might be interesting for future research.

For convex optimization, oracles based on linear optimization
have been studied extensively and lower bounds on the number of
queries are typically obtained by observing that each iteration adds 
only one vertex at a time
\cite{jaggi2013revisiting,lan2013complexity,ghPlaying2013}.
These oracles are typically weaker than general local oracles. 

The study of oracle complexity started with
the seminal work \cite{Nemirovski:1983},
where worst-case complexity is determined
up to a constant factor for several function families.
(see also \cite{Nemirovski:1994,Nesterov:2004} for
alternative proofs and approaches).
Interestingly,
these bounds were extended to randomized oracle complexity
at the price of a logarithmic
multiplicative gap \cite[4.4.3 Proposition 2]{Nemirovski:1983}. 
The proof of the latter result is somewhat technical, involving
various reductions from randomized to deterministic algorithms,
together with a union bound on the trajectories of the 
algorithms; this use of the union bound is essentially the source
of the logarithmic gap. On the other hand, our arguments are robust to 
randomization, due to our focus on the distributional setting. 
It is nevertheless important to emphasize that most of the 
function families we employ for lower bounds are either borrowed 
from or inspired by constructions in \cite{Nemirovski:1983}.

Recently, the study of lower bounds for stochastic oracles 
has become a widely popular topic, motivated by their connections
with machine learning. Such oracles were first studied in 
\cite{Nemirovski:1983}, and for recent lower bounds we refer to
\cite{Raginsky:2011,Agarwal:2012}. In this work we do not 
consider stochastic oracles.

%and also to stochastic oracles.
%In this paper, we do not consider stochastic oracles;
%for recent lower bounds, see \cite{Raginsky:2011,Agarwal:2012}.
%However,
%we will employ
%function families similar to those used in \cite{Nemirovski:1983}
%for lower bounds on worst-case oracle complexity.

An interesting result in \cite{Sridharan} provides a general 
(worst-case) lower complexity bound for Lipschitz convex 
minimization in terms of fat 
shattering dimension of the class of linear functionals where
the subgradients lie.
%numbers for the worst case. 
As expected, our 
lower bounds coincide with these fat shattering numbers, but hold 
under more general assumptions, namely the distributional setting. 
In a related note, the lower bounds of this paper have been 
extended to handle more general
(a.k.a. non-standard) settings; for these results we refer to 
\cite[Corollary 3.8.1]{GuzmanPhD}.

While our lower bounds are obtained in a fashion somewhat similar to those in statistical
minimax theory, the key in our approach is estimating
what is learned from each obtained subgradient given what has been
learned from previous subgradients---in statistical
minimax theory, we typically take (random!) samples drawn i.i.d (see
\cite{Raginsky:2011} for a detailed discussion).

\subsection*{Contribution}

We unify lower bounding techniques for convex nonsmooth optimization
by identifying a common source of hardness and introducing an
emulation mechanism that allows us to reduce different convex 
optimization settings to this setup. Our arguments are surprisingly 
simple, allowing for a unified treatment. 

\begin{description}

\item[Information-theoretic framework.]

We present an information-theoretic framework to lower bound
the oracle complexity of any type of oracle problem.
The key insight is that
if the information content of the oracle answer to a query is
low on average,
then this fact alone is enough for establishing
a strong lower bound on
both the distributional and the high probability complexity.

\item[Common source of hardness.]

Our base problem is learning a hidden string via guessing,
called the \emph{String Guessing Problem (SGP)}.
In Proposition~\ref{prop:string-via-oracle} we
establish a strong lower bound on the distributional and high 
probability oracle complexity of the string guessing problem, even for
algorithms with bounded error.
These bounds on the oracle complexity are established via
a new information-theoretic framework for iterative oracle-based
%query-based 
 algorithms. 

We then introduce a special reduction mechanism,
an \emph{emulation} in Definition~\ref{def:oracle-emulation},
rewriting algorithms and oracles between different problems,
see Lemma~\ref{lem:emulator-complexity}.
This will be the common framework for our lower bounds.

\item[First lower bounds for distributional and high-probability
  complexity for all local oracles.]

First, we establish lower bounds on the complexity for
a simple class of first-order local oracles for
Lipschitz-continuous convex functions
both on the \(L^{\infty}\)-box
in Theorem~\ref{th:box-nonsmooth-lower-bound}
and on the \(L^{p}\)-ball in Theorem~\ref{thm:large-scale-ball}.

In Section~\ref{sec:oracle_ind_LB} we extend
all lower bounds
in Theorem~\ref{perturb-large-scale}
to arbitrary local oracles
by using \emph{random perturbation}, 
instead of adaptive perturbation 
as done for worst-case lower bounds. 
%as in \cite{Nemirovski:1983}. 
A key technical aspect is what we call
the Lemma of unpredictability
(Lemmas~\ref{lem:unpredictability-ball-case} and
\ref{lem:unpredictability-box-case}),
which asserts that
with probability 1 arbitrary local
oracles are not more informative than the simple oracles
studied in Sections \ref{sec:oracle-compl-box} and 
\ref{sec:large-scale-compl} when adding random
perturbations.

%These 
The resulting bounds match classical lower bounds on worst-case complexity
(see Figure~\ref{fig:smallVsLarge}),
but established for distributional oracle complexity,
i.e., average case complexity,
and high-probability oracle complexity. 
Finally, our analysis extends to bounded-error algorithms:
even if the algorithm is allowed to provide erroneous answer
with a bounded probability
(e.g. discard a bounded subset of instances,
or be correct only with a certain probability
on every instance),
essentially the same lower bounds
hold. 

\item[Closing the gap between randomized and worst-case oracle
  complexity.] 

In the case of the $L^{\infty}$-box as well as the \(L^p\)-ball
for \(1 \leq p < \infty\),
our bounds show that all four complexity measures coincide,
namely, high-probability, distributional, randomized, and
worst-case complexity.
This not only simplifies the proofs in \cite{Nemirovski:1983}
for randomized complexity, but also closes the gap between
worst-case and randomized complexity (\cite[4.4.3 Proposition
2]{Nemirovski:1983}).
\end{description}

\begin{figure}[h!]
  \centering
  \small
  %% Adapted from \cite[p 117]{Nemirovski:1994}
  \begin{tikzpicture}[x=4mm, y=4mm]
    %% Here p=2
    %% Coordinate axis
    %% y-axis
    \draw[-latex] (0,0) -- ++(up:10)
    node[left]{\(\Compl_{\mathcal{D}}(\mathcal{F},\varepsilon)\)};
    %% x-axis
    \draw[-latex] (0,0) -- ++(right:15)
    %%% anchor def of q to the tip of the axis
    node[above left=2]{\(r \coloneqq \max\{p,2\}\)}
    %%% x-axis label
    node[below]{\(1 / \varepsilon\)};
    %% Plot: n=4
    \draw[label={\(\frac{1}{\varepsilon^{r}}\)}]
    %% [1, 2 sqrt(n)]: 1/eps^2 = x^2
    %% fake a bit, to start at (1,1) = (1/4, 1/4)
    plot[domain=1/4:4, smooth, samples=4]
    (\x, {(15/16 + \x^2) / 4})
    coordinate(break-square)
    [namey={\(n\)}, namex={\(\sqrt[r]{n}\)}]
    ;
    %% [2 sqrt(n), 2n^(3/4)]
    \draw[dotted]
    %% ad-hoc transition
    plot[domain=4:6, smooth, samples=4] (\x,
    {((1 - tanh(3 * (\x - 4.5))) / 2) * ((15/16 + \x^2) / 4)
      +
      ((1 + tanh(3 * (\x - 4.5))) / 2) * 3 * ln(\x)})
    coordinate(break-linear)
    [namey={\(n \log n^{1/r + \delta}\)},
     namex={\(n^{1/r + \delta}\)}];
    %% [2n^(3/4),-): n ln(1/eps) = n ln x
    \draw[label={\(n \log \frac{1}{\varepsilon}\)}]
    plot[domain=6:15, smooth, samples=4]
    (\x, {3 * ln(\x)})
    ;
  \end{tikzpicture}
  \caption{\label{fig:smallVsLarge}Distributional complexity
    as a function of $1/\varepsilon$
    for the \(L^{p}\)-ball, \(1 \leq p < \infty\).}
\end{figure}
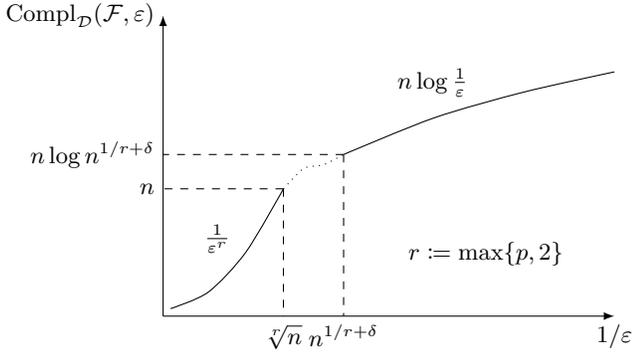

\section{Preliminaries}
\label{sec:preliminaries}

\subsection{Convex functions and approximate solutions}
In the following,
let \(X\) be a convex body in \(\RR^{n}\),
i.e., a full dimensional compact convex set.
We denote 
by \(B_p(x,r)\) the ball in \(\RR^n\) centered at \(x\)
with radius \(r\) in the \(L^p\)
norm, where \(1 \leq p \leq \infty\).
Recall that \(B_{\infty}(x,r)=\prod_{i=1}^n[x_i-r,x_i+r]\).
Let \(e_i\) denote
the \(i\)-th coordinate vector in \(\RR^n\).

Recall that a function $f \colon X\to \RR$ is \emph{convex} if for
all $x,y\in X$ and $0\leq \lambda\leq 1$,
\[ f(\lambda x+(1-\lambda)y) \leq \lambda f(x)+(1-\lambda)f(y).\]
Recall that $f$ is \emph{subdifferentiable} at $x\in X$ if there
exists $g\in\RR^n$ such that for all $y\in X$, 
\[ f(x)+\langle g, y-x\rangle \leq f(y).\]
In this case, we say that $g$ is a \emph{subgradient} of $f$ at $x$, and
the set of all subgradients of $f$ at $x$ is called the \emph{subdifferential},
denoted by $\partial f(x)$. It should be noted that when $f$ is
differentiable at $x$, the subdifferential is a singleton, namely 
$\partial f(x)=\{\nabla f(x)\}$. 
The connection to the differentiable case 
leads to the interpretation of a subgradient as a proxy for the local
behavior of $f$ around $x$, although in the non-differentiable case
the subgradient only provides an underestimate for the function (see 
Figure~\ref{fig:cvx_fc}).
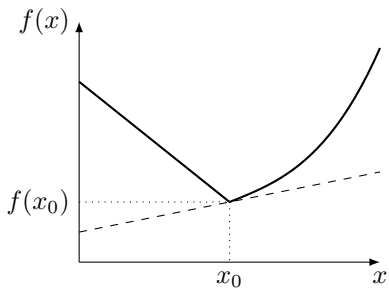
\begin{figure}[h!]
\centering
\begin{tikzpicture}[x=4mm, y=4mm]
\draw[-latex] (0,0) -- ++(up:8)
    node[left]{\(f(x)\)};
    %% x-axis
    \draw[-latex] (0,0) -- ++(right:10)
    %%% x-axis label
    node[below]{\(x\)};
    \draw[thick] (0,6) -- (5,2);
    \draw[thick,domain=5:10] plot (\x,{0.4*(\x-5)+2+0.025*(\x-5)*(\x-5)*(\x-5)});
   \draw[dashed,domain=0:10] plot(\x, {2+(0.2*(\x-5)});
   \draw[dotted] (0,2) node [left] {\(f(x_0)\)} -- (5,2);
   \draw[dotted] (5,2) -- (5,0) node [below] {\(x_0\)};
\end{tikzpicture}
 \caption{\label{fig:cvx_fc} Graph of a convex function in solid thick
 line. A subgradient at $x_0$ in dashed line.}
\end{figure}
Finally, recall that a function $f \colon X\to \RR$ is
Lipschitz-continuous with
Lipschitz constant $L$ with respect to a norm $\|\cdot\|$
if for all $x,y \in X$,
\[\size{f(x) - f(y)} \leq L \norm{x-y}.\]

Let \(\mathcal{F}\) be a family of real valued,
Lipschitz-continuous convex functions on \(X\)
with Lipschitz constant \(L\) with respect to $\|\cdot\|$.
For each $f\in \mathcal{F}$,
let $f^{\ast} \coloneqq\min_{x\in X} f(x)$.
Given an accuracy level \(\varepsilon > 0\),
an \emph{\(\varepsilon\)-minimum of \(f\)}
is a point \(x \in X\) satisfying
\(f(x) < f^{\ast} + \varepsilon\).
The set of \(\varepsilon\)-minima
will be denoted by \( \mathcal{S}_{\varepsilon}(f)\).

In general, an \(\varepsilon\)-minimum
need not identify \(f\) uniquely.
However, it simplifies the analysis
when \(\varepsilon\)-minima
identify the function instance, as this
makes optimization equivalent to learning
the instance.
We call this the packing property:
\begin{definition}[Packing property]
  \label{def:packing}
  A function family \(\mathcal{F}\) satisfies the \emph{packing property}
  for an accuracy level \(\varepsilon\),
  if no two different members \(f,g \in \mathcal{F}\)
  have common \(\varepsilon\)-minima, i.e., 
  if \(\mathcal{S}_{\varepsilon}(f) \cap \mathcal{S}_{\varepsilon}(g)
    = \emptyset.\)
\end{definition}

\subsection{Oracles and Complexity}

We analyze the distributional complexity of
approximating solutions in convex optimization
under the standard black box oracle model,
where algorithms have
access to the instance \(f\) only by querying 
an oracle \(\mathcal{O}\). Our prototypical 
example is minimization of a convex
function by first-order methods: For this we prescribe  
a class of convex functions $\mathcal{F}$
with domain $X$, and a target accuracy $\varepsilon>0$.
First-order methods
are based on sequentially querying feasible points on 
$X$, computing the value and subgradient
of $f$ at these points, and using this information to find an
$\varepsilon$-minimum. This computational paradigm includes 
most known methods for continuous optimization, such as 
Subgradient Descent, Mirror-Descent, Center of Gravity, and the 
Ellipsoid Method,
among others. The motivation behind oracle complexity is to
establish which methods 
are optimal: If the cost of implementing each step of the method 
is not too high, its number of iterations (that is,
the number of oracle calls) is a reasonable
proxy for the overall complexity.
%\emph{Oracle complexity} is measured by
%the least number of queries \(T_{A}(f, \varepsilon)\)
%required to find an \(\varepsilon\)-minimum.

Let us now introduce the model. 
Given a class of convex functions $\mathcal{F}$ with domain $X$,
an oracle $\mathcal{O}$ provides partial information
\(\mathcal{O}_{f}(x)\) about an
unknown instance \(f\) from $\mathcal{F}$.
When the instance \(f\) is clear from the context
we shall omit the subscript \(f\).

The considered oracles  $\mathcal{O}$ are \emph{local}. 
We say that $\mathcal{O}$ is a local oracle
if for all $f_{1}, f_{2} \in \mathcal{F}$
that are equal in a 
neighborhood of \(x\),
we have $\mathcal{O}_{f_{1}}(x) = \mathcal{O}_{f_{2}}(x)$.
An important case is a \emph{first-order} oracle,
which answers a query \(x \in X\) by
\(\mathcal{O}_f(x)=(f(x),g)\), where
$f(x)$ is the function value and $g\in\partial f(x)$ is a 
subgradient of \(f\) at \(x\). 
Note however that not every first-order oracle is local:
at non-differentiable points of \(f\) a non-local oracle can
choose between various subgradients
taking into account the whole function \(f\),
thereby e.g., encoding an \(\varepsilon\)-minimum. 
The requirement of locality allows us to rule out such 
unnatural situations.

Let \(\mathcal{A(O)}\) denote the set of
deterministic algorithms based on oracle \(\mathcal{O}\).
Given an algorithm $A\in \mathcal{A(O)}$, an instance 
$f\in\mathcal{F}$, and target accuracy
$\varepsilon>0$, we denote \(T_{A}(f, \varepsilon)\)
the number of queries $A$ makes in order to reach
an $\varepsilon$-minimum. This way,
the \emph{worst-case} oracle complexity is
defined as
\[\Compl_{\mathcal{WC}}(\mathcal{O},\mathcal{F},\varepsilon)
\coloneqq \adjustlimits \inf_{A\in\mathcal{A(O)}}
\sup_{f\in\mathcal{F}} T_A(f,\varepsilon).\]
Following \cite{Nemirovski:1983},
\emph{randomized complexity} is defined as
\[\Compl_{\mathcal{R}}(\mathcal{O},\mathcal{F},\varepsilon)
\coloneqq
\adjustlimits
\inf_{A\in \Delta(\mathcal{A(O)})}
\sup_{f\in\mathcal{F}} \expectation(A){T_{A}(f,\varepsilon)},\]
where \(\Delta(\mathcal{B})\) is the set of
probability distributions on the set \(\mathcal{B}\). 
The interpretation of this notion of complexity is that
randomized algorithms have the additional power of
private sources of randomness, and can alternatively be 
seen as a mixture of deterministic algorithms.

The measure we will bound in our work is the even
weaker notion of \emph{distributional complexity}
\[\Compl_{\mathcal{D}}(\mathcal{O},\mathcal{F},\varepsilon)
\coloneqq
\adjustlimits
\sup_{F \in \Delta(\mathcal{F})} \inf_{A\in \mathcal{A(O)}}
\expectation(F){T_A(F,\varepsilon)},\] 
leading to stronger lower bounds: Notice that in this case the 
algorithm has full knowledge about the instance distribution.
We will also bound the \emph{high-probability oracle
complexity} defined as
\begin{align*}
\begin{multlined}
\Compl_{\mathcal{HP}}(\mathcal{O},\mathcal{F},\varepsilon) 
\\\coloneqq
\adjustlimits
\sup_{F\in\Delta(\mathcal{F})} \inf_{A\in \mathcal{A(O)}}
\sup_{\tau : \probability(F){T_A(F,\varepsilon) \geq \tau} = 1-
  o(1)} \tau, 
\end{multlined}
\end{align*}
i.e., it is the number \(\tau\) of required queries
that any algorithm needs for the worst distribution
with high probability.
It is easily seen that
\begin{multline*}
\Compl_{\mathcal{HP}}(\mathcal{O},\mathcal{F},\varepsilon)
\leq
\Compl_{\mathcal{D}}(\mathcal{O},\mathcal{F},\varepsilon)\\
\leq \Compl_{\mathcal{R}}(\mathcal{O},\mathcal{F},\varepsilon)
\leq \Compl_{\mathcal{WC}}(\mathcal{O},\mathcal{F},\varepsilon),
\end{multline*}
but it is open
for which families \(\mathcal{F}\) this
inequality chain is tight, e.g., whether Yao's min-max principle 
applies (see e.g., \cite{AroraBarakBook}),
as both \(\mathcal{F}\) and \(\mathcal{A}\) might be
\emph{a priori} infinite families. However it is known that
worst-case and randomized complexity are
equal up to a factor logarithmic in the dimension
for several cases (see \cite[4.4.3 Proposition
2]{Nemirovski:1983}). 

\subsection{Algorithm-oracle communication and string operations}

For a given oracle-based (not necessarily minimization) algorithm,
we record the communication between the algorithm and the oracle.
Let \(Q_{t}\) be the \(t\)-th query of the algorithm
and \(A_{t}\) be the \(t\)-th oracle answer.
Thus \(\Pi_{t} \coloneqq (Q_{t}, A_{t})\) is
the \(t\)-th query-answer pair.
The full transcript of the communication is denoted by
$\Pi = (\Pi_1,\Pi_2,\dotsc)$, and for given \(t\geq 0\)
partial transcripts are defined as \(\Pi_{\leq t}
\coloneqq (\Pi_1,\dotsc,\Pi_t)\) and \(\Pi_{< t} \coloneqq
(\Pi_1,\dotsc,\Pi_{t-1})\).
By convention, \(\Pi_{< 1}\) and \(\Pi_{\leq 0}\) are
empty sequences.

As we will index functions by strings,
let us introduce the necessary string operations.
Let \(s \in \binSet^*\) be a binary string, then
\(\flipCut{i}{s}\) denotes the string obtained from \(s\) by flipping
the \(i\)-th bit and deleting all bits following the \(i\)-th one.
Let \(s \sqsubseteq t\) denote that \(s\) is a prefix of \(t\)
and \(s \parallel t\) denote that neither is a prefix of the other.
As a
shorthand let \(\prefix{l}{s}\)
be the prefix of \(s\) consisting of the first \(l\) bits.
We shall write \(s0\) and \(s1\) for the strings obtained by
appending a \(0\) and \(1\) to \(s\), respectively.
Furthermore, the
empty string is denoted by \(\bot\). In the following we use the
shorthand notation \([n] \coloneqq \set{1,\dots,n}\).

\subsection{Information Theory}

Notions from information theory are standard as defined in
\cite{cover2006elements}; we recall here those we need later.
From now on, \(\log(.)\) denotes the binary logarithm and
capital letters will typically represent random variables or events.
We can describe an event \(E\) as a random variable by the 
indicator function \(I(E)\), which takes value 1 if \(E\) happens, 
and 0 otherwise.

The \emph{entropy} of a discrete random variable
\(A\) is
\[\entropy{A} \coloneqq - \sum_{a \in \range(A)} \probability{A=a}
\log \probability{A=a}.\]
This definition extends naturally to \emph{conditional entropy}
\(\entropy[B]{A}\)
by using conditional distribution and taking
expectation, i.e., \(\entropy[B]{A} =
\sum_{b} \probability{B = b}{\entropy[B = b]{A}}\).

\begin{fact}[Properties of entropy]
  \mbox{}
  \begin{description}
  \item[Bounds]
    \(0 \leq \entropy{A} \leq \log \size{\range(A)}\)

    \(\entropy{A} = \log \size{\range(A)}\)
    if and only if \(A\) is uniformly distributed.
  \item[Monotonicity]
    \(\entropy{A} \geq \entropy[B]{A}\);
  \end{description}
\end{fact}

The notion of \emph{mutual information} defined as
\(\mutualInfo{A}{B} \coloneqq \entropy{A} - \entropy[B]{A}\)
of two random variables \(A\) and \(B\) 
captures how much information about a \lq{}hidden\rq{} \(A\)
is leaked by observing \(B\).
Sometimes \(A\) and \(B\) are a collection of variables,
then a comma is used to separate the
components of \(A\) or \(B\), and a semicolon to separate
\(A\) and \(B\) themselves:
e.g., \(\mutualInfo{A_{1}, A_{2}}{B} = \mutualInfo{(A_{1}, A_{2})}{B}\).
Mutual information is a symmetric quantity and
naturally extends to \emph{conditional mutual information}
\(\mutualInfo[C]{A}{B}\) as in the case of entropy. Clearly, 
\(\entropy{A} = \mutualInfo{A}{A}\).

\begin{fact}[Properties of mutual information]
  \mbox{}
  \begin{description}
  \item[Bounds]
    If \(A\) is a discrete variable, then
    \(0 \leq \mutualInfo{A}{B} \leq \entropy{A}\)
  \item[Chain rule]
    \(\mutualInfo{A_1,A_2}{B} =
    \mutualInfo{A_1}{B} + \mutualInfo[A_1]{A_2}{B}\).
  \item[Symmetry] \(\mutualInfo{A}{B} = \mutualInfo{B}{A}\).
  \item[Independent variables]
    The variables \(A\) and \(B\) are independent
    if and only if
      \(\mutualInfo{A}{B} = 0\).
  \end{description}
\end{fact}

\section{Source of hardness and oracle emulation}
\label{sec:general-lower-bounds}

We provide a general method to lower bound
the number of queries of an algorithm that 
identifies a hidden random variable. This 
method is based on information theory and 
will allow us to lower bound the
distributional and high probability oracle
complexity, even for bounded-error algorithms.
We apply this technique to the problem of
identifying a random binary string, which 
we call the String Guessing Problem. 
Finally, we introduce an oracle emulation
technique, that will allow us to compare the
complexity of different oracles solving the same
problem.

\subsection{Information-theoretic lower bounds}

We consider an unknown instance \(F\)
that is randomly chosen from
a finite family \(\mathcal{F}\) of instances.
For a given algorithm querying an 
oracle \(\mathcal{O}\), let \(T\) be the
number of queries the algorithm asks to determine the
instance. Of course, the number \(T\) may depend on 
the instance, as algorithms can adapt their queries
according to the oracle answers. However, we assume 
that \(T<\infty\) almost surely, i.e., we require 
algorithms to almost always terminate
(this is a mild assumption as \(\mathcal{F}\) is finite).

Algorithms are allowed to have an error probability 
bounded by \(P_e\), i.e., the algorithm is only required 
to return the correct answer with
probability \(1-P_e\) \emph{across all instances}. The latter
statement is important as both, being perfectly correct on a \(1-P_e\)
fraction of the input and outputting garbage in \(P_e\) cases, as well
as providing the correct answer for \emph{each instance} with
probability \(1-P_e\), are admissible here. 

For bounded-error algorithms,
the high-probability complexity is the required number of queries
to produce a correct answer with probability \(1 - P_{e} - o(1)\).
This adjustment is justified, as a wrong answer is allowed with
probability \(P_{e}\).

\begin{lemma}
\label{lem:EntropyLB}
Let $F$ be a random variable with finite range \(\mathcal{F}\).
For a given algorithm determining \(F\)
via querying an oracle, with error probability bounded by \(P_{e}\),
suppose that the useful information of
each oracle answer is bounded, i.e., for some constant $C>0$, we have
\begin{equation}\label{EntUnifBd}
  \mutualInfo[\Pi_{<t}, Q_{t}, T\geq t]{F}{A_t} \leq C,
  \qquad t\geq 0.
\end{equation}
Then, the distributional oracle complexity of the algorithm is
lower bounded by
\[ \expectation{T} \geq
   \frac{\entropy{F} -
     \bentropy{P_{e}} - P_{e} \log \size{\mathcal{F}}}{C}. \]
Moreover, for all \(t\) we have
\begin{equation*}
  \probability{T < t}
  \leq
  \frac{\bentropy{P_{e}} + P_{e} \log \size{\mathcal{F}} + C t}
  {\entropy{F}}.
\end{equation*}
In particular,
if \(F\) is uniformly distributed,
then
\(\probability{T = \Omega(\log \size{\mathcal{F}})}
= 1 - P_{e} - o(1)\).

\begin{proof}
By induction on \(t\) we will first prove the following claim
\begin{equation}
  \label{eq:mutualInfo-split-finite}
\begin{multlined}
  \mutualInfo{F}{\Pi} 
  =
    \sum_{i=1}^{t} \mutualInfo[\Pi_{<i},T\geq i]{F}{\Pi_i}
    \probability{T\geq i} \\
  + \mutualInfo[\Pi_{\leq t}, T \geq t]{F}{\Pi}
  \probability{T\geq t}.
\end{multlined}
\end{equation}

The case \(t=0\) is obvious.
For \(t > 0\),
note that
the event \(T = t\) is independent of \(F\)
given \(\Pi_{\leq t}\),
as at step \(t\)
the algorithm has to decide whether to continue
based solely on the previous oracle answers
and private random sources.
If the algorithm stops, then \(\Pi = \Pi_{\leq t}\).
Therefore,
\begin{equation*}
\begin{aligned}
  &\mutualInfo[\Pi_{\leq t}, T \geq t]{F}{\Pi} \\
 =&
  \mutualInfo[\Pi_{\leq t}, T \geq t]{F}{\Pi, I(T=t)}
  \\
  =&
  \underbrace{\mutualInfo[\Pi_{\leq t}, T \geq t]{F}{I(T=t)}}_{=0} +
  \mutualInfo[\Pi_{\leq t}, I(T=t), T \geq t]{F}{\Pi}
  \\
  =&
  \underbrace{\mutualInfo[\Pi_{\leq t}, T = t]{F}{\Pi}}_{= 0\text{, as \(\Pi_{\leq t} = \Pi\)}}
  \probability[T \geq t]{T = t} \\
  &\!+\!
  \mutualInfo[\Pi_{\leq t}, T \geq t+1]{F}{\Pi}
  \probability[T \geq t]{T \geq t+1}\\
  =&
  \Big(
  \mutualInfo[\Pi_{<t+1}, T \geq t+1]{F}{\Pi_{t+1}}\\
  &\!+\!
  \mutualInfo[\Pi_{\leq t+1},T \geq t+1]{F}{\Pi}
  \Big)
  \probability[T \geq t]{T \geq t+1},
\end{aligned}
\end{equation*}
obtaining the identity
\begin{equation*}
\begin{multlined}
\mutualInfo[\Pi_{\leq t}, T \geq t]{F}{\Pi} \probability{T \geq t}
= \\
\left(
  \mutualInfo[\Pi_{<t+1}, T \geq t+1]{F}{\Pi_{t+1}}
  \!+\!
  \mutualInfo[\Pi_{\leq t+1},T \geq t+1]{F}{\Pi}
  \right) \\
\cdot\probability{T \geq t+1},
\end{multlined}
\end{equation*}
from which the induction follows.

Now, in \eqref{eq:mutualInfo-split-finite}
by letting \(t\) go to infinity,
\(\probability{T\geq t}\)
will converge to \(0\),
while
\(\mutualInfo[\Pi_{\leq t}, T \geq t]{F}{\Pi}\)
is bounded by \(\entropy{F}\), proving that
\begin{equation}
  \label{eq:mutualInfo-split-infinite}
  \mutualInfo{F}{\Pi}
  =
  \sum_{i=1}^{\infty} \mutualInfo[\Pi_{<i},T\geq i]{F}{\Pi_i}
  \probability{T\geq i}.
\end{equation}
Note that \(Q_{i}\) is chosen solely based on \(\Pi_{<i}\), and is 
conditionally independent of \(F\). Therefore, by the chain rule,
\(
  \mutualInfo[\Pi_{<i},T\geq i]{F}{\Pi_i} =
  \mutualInfo[\Pi_{<i}, Q_{i}, T\geq i]{F}{A_i}.
\)
Plugging this equation into \eqref{eq:mutualInfo-split-infinite}, we obtain
\begin{eqnarray*}
  \mutualInfo{F}{\Pi}
  &=&
  \sum_{i=1}^{\infty} \mutualInfo[\Pi_{<i}, Q_{i}, T\geq i]{F}{A_i}
  \probability{T\geq i} \\
  &\leq& C \sum_{i=0}^{\infty} \probability{T\geq i}\\
  &=& C \cdot \expectation{T}.
\end{eqnarray*}
Finally, as the algorithm determines \(F\)
with error probability at most \(P_{e}\),
Fano's inequality \cite[Theorem~2.10.1]{cover2006elements} applies
\begin{equation}
  \label{eq:Fano}
  \entropy[\Pi]{F}
  \leq \bentropy{P_{e}} + P_{e} \log \size{\mathcal{F}}.
\end{equation}
We therefore obtain
\begin{equation*}
  \entropy{F} = \entropy[\Pi]{F} + \mutualInfo{F}{\Pi}
  \leq \bentropy{P_{e}} + P_{e} \log \size{\mathcal{F}}
  + C \cdot \expectation{T},
\end{equation*}
and therefore
\begin{equation*}
  \expectation{T}
  \geq
  \frac{\entropy{F} -
    \bentropy{P_{e}} - P_{e} \log \size{\mathcal{F}}}{C}
  ,
\end{equation*}
as claimed.

We will now establish concentration for the number of required
queries. For this we reuse \eqref{eq:mutualInfo-split-finite},
the split-up of information up to query \(t\):
\begin{equation*}
 \begin{split}
  \mutualInfo{F}{\Pi} =&
    \sum_{i=1}^{t} \mutualInfo[\Pi_{<i}, T \geq i]{F}{\Pi_i}
    \probability{T\geq i} \\
  &+ \mutualInfo[\Pi_{\leq t}, T \geq t]{F}{\Pi}
  \probability{T\geq t}
  \\
  =&
    \sum_{i=1}^{t}
    \underbrace{\mutualInfo[\Pi_{<i}, Q_{i}, T\geq
      i]{F}{A_i}}_{\leq C}
    \probability{T\geq i} \\
  &+ \underbrace{\mutualInfo[\Pi_{\leq t}, T \geq t]{F}{\Pi}}_{\leq \entropy{F}}
  \probability{T\geq t}
  \\
  \leq&\,\,  C t + \entropy{F} \probability{T\geq t},
 \end{split}
\end{equation*}
which we combine with \eqref{eq:Fano}:
\begin{equation*}
\begin{multlined}
  \entropy{F} = \entropy[\Pi]{F} + \mutualInfo{F}{\Pi} \\
  \leq \bentropy{P_{e}} + P_{e} \log \size{\mathcal{F}}
  + C t + \entropy{F} \probability{T\geq t},
\end{multlined}
\end{equation*}
and therefore
\begin{equation*}
  \probability{T < t}
  \leq
  \frac{\bentropy{P_{e}} + P_{e} \log \size{\mathcal{F}} + C t}
  {\entropy{F}}.
\end{equation*}
Specializing to uniform distributions provides
the last claim of the Lemma.
\end{proof}
\end{lemma}

\subsection{Identifying binary strings}
\label{sec:ident-binary-strings}

For a fixed length \(M\) we consider the problem of identifying 
a hidden string \(S \in \{0,1\}^M\) picked uniformly at random.
The oracle \(\mathcal{O}_{S}\) accepts
queries for any part of the string.
Formally, a query is a pair \((s, \sigma)\),
where \(s\) is a string of length at most \(M\),
and \(\sigma \colon [\card{s}] \to [M]\) is an embedding
indicating an order of preference.
The intent is to ask whether \(S_{\sigma(k)} = s_{k}\) for all \(k\).
The oracle will
reveal the smallest \(k\) so that \(S_{\sigma(k)} \neq
s_k\) if such a
\(k\) exists or will assert correctness of
the guessed part of the string.
More formally we have:

\begin{orac}[String Guessing Oracle \(\mathcal{O}_{S}\)]
  \label{orac:idStringConf}
  \mbox{}% don't start list at heading
  \begin{description}[font={\normalfont\itshape}, nosep]
  \item[Query:] A string \(s \in \{0,1\}^{\leq M}\)
    and an injective function \(\sigma \colon [\card{s}] \to [M]\).
  \item[Answer:] Smallest \(k \in \N\) so that
    \(S_{\sigma(k)} \neq s_k\) if it exists,
    otherwise \texttt{EQUAL}.
  \end{description}
\end{orac}

From Lemma \ref{lem:EntropyLB} we establish an expectation
and high probability lower bound on the number of queries,
even for bounded error algorithms.
The key is that the oracle does not reveal any information
about the bits after the first wrongly guessed bit,
not even involuntarily.

\begin{proposition}[String Guessing Problem]
  \label{prop:string-via-oracle}
  Let \(M\) be a positive integer,
  and \(S\) be a uniformly random binary string of length \(M\).
  Let \(\mathcal{O}_{S}\) be the String Guessing Oracle
  (Oracle~\ref{orac:idStringConf}).
  Then for any bounded error algorithm
  having access to \(S\) only through \(\mathcal{O}_{S}\),
  the expected number of queries required to identify \(S\)
  with error probability at most \(P_{e}\)
  is at least \([(1 - P_{e}) M - 1] / 2\).
  Moreover,
  \(\probability{T = \Omega(M)} = 1 - P_{e} - o(1)\),
  where \(T\) is the number of queries.
\begin{proof}
We will prove the following claim by induction: At any step \(t\),
given the partial transcript \(\Pi_{< t}\),
some bits of \(S\) are totally determined,
and the remaining ones are still uniformly distributed.
The claim is obvious for \(t = 0\). Now
suppose that the claim holds for some \(t-1\geq 0\).
The next query \(Q_t \coloneqq (s;\sigma)\) is independent of \(S\)
given \(\Pi_{< t}\).
Let us fix \(\Pi_{< t}\) and \((s;\sigma)\),
and implicitly condition on them
until stated otherwise.
We differentiate two cases. 
\begin{casesblock}
\begin{case}{The oracle answer is \texttt{EQUAL}.}
This is the case if and only if \(s_{\ell} =
S_{\sigma({\ell})}\) for all \(\ell \in [\card{s}]\). Thus 
the oracle answer reveals the bits \(\{S_{\sigma(\ell)} \mid \ell \in
[\card{s}]\}\), actually determining them.
\end{case}
\begin{case}{The oracle answer is \(k\).}
This is the case if and only if
\(s_{j} = S_{\sigma(j)}\) for all \(j < k\) and \(s_{k} \neq S_{\sigma(k)}\).
Thus the oracle answer reveals \(\{S_{\sigma(\ell)} \mid \ell \in
[k]\}\) (the \(k\)-th bit by flipping), determining them.
\end{case}
\end{casesblock}
In both cases,
the answer is independent of the other bits,
therefore
the ones among them,
which are not determined by previous oracle answers,
remain uniformly distributed
and mutually independent.
This establishes the claim for \(\Pi_{t}\), finishing the induction.

We extend the analysis to estimate the mutual information
of \(S\) and the oracle answer \(A_{t}\).
We keep \(\Pi_{<t}\) and \(Q_t\) fixed,
and implicitly assume \(T \geq t\),
as otherwise \(Q_{t}\) and \(A_{t}\) don't exist.
For readability, we drop the conditions in the computations
below; all quantities are to be considered conditioned
on \(\Pi_{<t}\), \(Q_t\) provided \(T \geq t\).

Let \(m \coloneqq \entropy{S}\) be
the number of undetermined bits just before query \(t\).
Let \(K\) be the number of additionally determined bits due to query
\(Q_t\) and oracle answer \(A_t\), hence
obviously \[\entropy[A_{t}]{S} = \expectation{m - K}.\]

The analysis above shows that for all \(k \geq 1\),
a necessary condition for \(K \geq k\) is
that \(s_{j} = S_{\sigma(j)}\) for
the \(k-1\) smallest \(j\)
with \(S_{\sigma(j)}\) not determined before query \(t\) and that
these \(k-1\) smallest \(j\) really exist.
The probability of this condition is \(1/2^{k-1}\)
(or \(0\) if there are not sufficiently many \(j\)) and so in any case
we have 
\begin{equation*}
  \probability{K \geq k} \leq \frac{1}{2^{k-1}}, \qquad k \geq 1.
\end{equation*}

Combining these statements we see that,
\begin{equation*}
\begin{multlined}
  \mutualInfo{S}{A_t}
  =
  \entropy{S} -
  \entropy[A_{t}]{S}
  =
  m -
  \expectation{m - K}\\
  =
  \expectation{K}
  =
  \sum_{i \in [m]}
  \probability{K \geq i}
  \leq \sum_{i \in [\infty]} \frac{1}{2^{i-1}} = 2,
\end{multlined}
\end{equation*}
with \(\Pi_{< t}, Q_{t}\) still fixed.

Now we re-add the conditionals, vary \(\Pi_{< t}, Q_{t}\),
and take expectation still assuming \(T \geq t\), obtaining
\begin{equation*}
  \mutualInfo[\Pi_{<t}, Q_{t}, T\geq t]{S}{A_t} \leq 2
\end{equation*}
where \(T\) is the number of queries.
By Lemma~\ref{lem:EntropyLB} we obtain
\(\expectation{T} \geq [(1 - P_{e}) M - \bentropy{P_{e}}] / 2
\geq [(1 - P_{e}) M - 1] / 2\) (the binary entropy
is upper bounded by 1)
and
\(\probability{T = \Omega(M)} = 1 - P_{e} - o(1)\),
as claimed.
\end{proof}
\end{proposition}

\subsection{Oracle emulation}
\label{sec:oracle-emulation}

In this section we introduce \emph{oracle emulation}, 
which is a special type of reduction from one oracle to 
another, both for the same family of instances. This 
reduction allows to transform algorithms based on one 
oracle to the other preserving their oracle complexity,
i.e, the number of queries asked.
The crucial result is Lemma~\ref{lem:emulator-complexity},
which we will apply to emulations of 
various convex optimization oracles
by the String Guessing Oracle \(\mathcal{O}_{S}\).

\begin{definition}[Oracle emulation]
  \label{def:oracle-emulation}
  Let \(\mathcal{O}_{1} \colon Q_{1} \to R_{1}\)
  and \(\mathcal{O}_{2} \colon Q_{2} \to R_{2}\)
  be two oracles
  for the same problem.
  An \emph{emulation} of \(\mathcal{O}_{1}\) by \(\mathcal{O}_{2}\)
  consists of
  \begin{enumerate}
  \item
    a query emulation function
    \(q\colon \dom \mathcal{O}_{1} \to \dom \mathcal{Q}_{2}\)
    (translating queries of \(\mathcal{O}_{1}\)
    for \(\mathcal{O}_{2}\)),
  \item
    an answer emulation function
    \(a\colon Q_{1} \times \operatorname{cod} \mathcal{O}_{2} \to
    \operatorname{cod} \mathcal{O}_{1}\)
    (translating answers back)
  \end{enumerate}
  such that \(\mathcal{O}_{1}(x) = a(x, \mathcal{O}_{2}(q(x)))\)  for
  all \(x\in Q_{1}\).
  Here \(\dom \mathcal{O}\) and \(\operatorname{cod} \mathcal{O}\)
  denote the set of queries and answers of oracle \(\mathcal{O}\).
\end{definition}

An emulation leads to a reduction, 
since emulated oracles are at least as complex
as the emulating ones.
%since emulating oracles are at least as complex as
%the oracles they emulate.

\begin{lemma}
  \label{lem:emulator-complexity}
  If there is
  an emulation of \(\mathcal{O}_{1}\) by \(\mathcal{O}_{2}\),
  then the oracle complexity of
  \(\mathcal{O}_{1}\) is at least that of \(\mathcal{O}_{2}\).
  Here oracle complexity can be
  worst-case, randomized, distributional, and high probability;
  all even for bounded-error algorithms.
\begin{proof}
Let \(A_{1}\) be an algorithm using \(\mathcal{O}_{1}\),
and let \(\mathcal{O}_{2}\) emulate \(\mathcal{O}_{1}\).
Let \(q\) and \(a\) be the query emulation function
and the answer emulation function, respectively.
We define an algorithm \(A_{2}\) for \(\mathcal{O}_{2}\)
simulating \(A_{1}\) as follows:
Whenever \(A_{1}\) asks a query \(x\) to oracle 
\(\mathcal{O}_{1}\), oracle \(\mathcal{O}_{2}\) is queried 
with \(q(x)\), and the simulated \(A_{1}\) receives as answer 
\(a(x, \mathcal{O}_{2}(q(x)))\)
(which is \(\mathcal{O}_{1}(x)\) by definition of the emulation).
Finally, the return value of the simulated \(A_{1}\) is returned.

Obviously, \(A_{2}\) makes the same number of queries as \(A_{1}\)
for every input, and therefore the two algorithms have the same
oracle complexity.
This proves that the oracle complexity of \(\mathcal{O}_{1}\)
is at least that of \(\mathcal{O}_{2}\).
\end{proof}
\end{lemma}

\section{Single-coordinate oracle complexity for the box}
\label{sec:oracle-compl-box}
In the following we will analyze a simple class of oracles,
called \lq{}single-coordinate\rq{},
closely mimicking the string guessing oracle.
Later, all results will be carried over to
general local oracles via perturbation
in Section \ref{sec:oracle_ind_LB}. 

In this Section and onwards, for convenience, we use the notation 
$\nabla f(x)$ for an arbitrary subgradient of $f$ at $x$. It should be noted 
however this is not necessarily the gradient, as the function may 
not be differentiable at the point.

\begin{definition}[Single-coordinate oracle]
  A first-order oracle \(\tilde{\mathcal{O}}\) is
  \emph{single-coordinate}
  if for all \(x \in X\) the subgradient 
   \(\nabla f(x)\) in its answer
  is the one supported on the least coordinate axis; i.e.,
  \(\nabla f(x) = \lambda e_i\)
  for the smallest \(1 \leq i \leq n\)
  with some \(\lambda \in \RR\).
\end{definition}

Choosing 
the smallest possible \(i\) corresponds to
choosing the first wrong bit by the String Guessing Oracle.
Not all function families possess a single-coordinate oracle,
but maximum of coordinate functions do,
and single-coordinate oracles are a natural choice
for them. From now on, we will 
denote single-coordinate oracles exclusively by
\(\tilde{\mathcal{O}}\).

We establish a lower bound on the
\emph{distributional} and \emph{high probability oracle 
complexity} for
nonsmooth convex optimization over \([-R,+R]^n\),
for single-coordinate oracles.

\begin{theorem}
  \label{th:box-nonsmooth-lower-bound}
  Let \(L, R > 0\).
  There exists a finite family \(\mathcal{F}\) of
  Lipschitz-continuous convex functions on the $L^{\infty}$-ball
  \(B_{\infty}(0,R)\) with Lipschitz constant $L$ in the
  $L^{\infty}$ norm, and a single-coordinate
  local oracle \(\tilde{\mathcal{O}}\),
  such that both the distributional
  and the high-probability oracle complexity
  for finding an \(\varepsilon\)-minimum
  of a uniformly random instance
  is \(  \Omega \left( n \log \frac{LR}{\varepsilon} \right) \).
  
  For bounded-error algorithms with error bound \(P_e\),
  the distributional complexity is
   \(  \Omega \left( (1-P_e)n \log \frac{LR}{\varepsilon}\right) \),
  and the high-probability complexity is
  \(\Omega\left( n \log \frac{LR}{\varepsilon} \right)\).
\end{theorem}

In the following we will restrict ourselves to the case \(L=R=1\),
as the theorem reduces to it via an easy scaling argument.
We start with the one dimensional case
in Section~\ref{subsec:one_dim_case}
for a simple presentation of the main ideas.
We generalize to multiple dimensions
in Section~\ref{sec:mult-case}
by considering maxima of coordinate functions,
thereby using the different coordinates to represent different portions
of a string.

\subsection{One dimensional case} \label{subsec:one_dim_case}

Let $X \coloneqq [-1,1]$,
we define recursively a function family $\mathcal{F}$ on \(X\),
which is inspired by the one in \cite[Lemma 1.1.1]{Nemirovski:1994}.
For an interval \(I = [a,b]\),
let \(I(t) \coloneqq a + (1 + t) (b - a) / 2\) denote
the \(t\)-point on \(I\) for \(-1 \leq t \leq 1\), e.g., \(I(-1)\) is
the left end point \(a\) of \(I\),
and
\(I(+1)\) is the right end point \(b\),
and \(I(0)\) is the midpoint.
Let \(I[t_{1}, t_{2}]\) denote the subinterval
\([I(t_{1}), I(t_{2})]\).
The family \(\mathcal{F} = \{f_{s}\}_{s}\) will be indexed by
binary strings \(s\) of length \(M\),
where \(M \in \N\) depends on the accuracy \(\varepsilon\)
and will be chosen later.
It is convenient to define \(f_{s}\) also for shorter strings,
as we proceed by recursion on the length of \(s\).
We also define intervals \(I_{s}\)
and breakpoints \(b_{l}\) of the range of the functions
satisfying the following properties:

\begin{enumerate}[label=(F-\arabic*)]
\item\label{item:intLength}
  The interval \(I_{s}\) has length \(2\cdot (1/4)^{\abs{s}}\).
  \\
  \textbf{Motivation}: allow a strictly nesting family.
\item\label{item:disjoint-if-not-prefix}
If \(s \parallel t\), then \(\interior{I_{s}} \cap \interior{I_{t}} = \emptyset\).
 If \(t \sqsubseteq s\),
  we have \(I_{s} \subseteq I_{t}\) (the \(I_s\) are
  nested intervals).
  \\
  \textbf{Motivation}: instances can be distinguished by their
  associated intervals. Captures packing property.
\item\label{item:family-increasing}
  \(f_{s} \geq f_{\prefix{l}{s}}\)
  with \(f_{s}(x) = f_{\prefix{l}{s}}(x)\) if \(x \in [-1,1]
  \setminus \interior{I_{\prefix{l}{s}} }\).
  \\
  \textbf{Motivation:}
  long prefix determines much of the function.
\item\label{item:family-values}
  The function \(f_{s}\) restricted to the interval \(I_{s}\)
  is of the form 
  \[    f_{s}(x) = b_{\abs{s}} - 2^{- 3 \abs{s}}
    + 2^{- \abs{s}} \abs*{x - I_{s}(0)} \qquad x \in I_{s},\]
  where
  \(b_{\abs{s}} = f_{s}(I_{s}(-1)) = f_{s}(I_{s}(+1))\)
  is the function value on the endpoints of \(I_{s}\).
  This is symmetric on \(I_s\) as
  \(I_{s}(0)\) is the midpoint of \(I_{s}\).
  \\
  \textbf{Motivation:} recursive structure:
  repeat absolute value function on small intervals.
\item\label{item:bounds}
  For \(t \sqsubseteq s\), we have \(f_{s}(x) < b_{\abs{t}}\)
  if and only if \(x \in \interior{I_{t}}\).
  \\
  \textbf{Motivation:}
  level sets encode substrings.
\end{enumerate}

\subsubsection*{Construction of the function family}

We start with the empty string \(\bot\), and
define \(f_{\bot}(x) \coloneqq \abs{x}\) and
\(I_{\bot} \coloneqq [-1, 1]\).
In particular, \(b_{0} = 1\). The further \(b_{k}\) we define via the recursion
\(b_{k + 1} \coloneqq b_{k} - 2 \cdot (1/4)^{k + 1} \cdot 2^{- k}\).

Given \(f_{s}\) and \(I_{s}\),
we define \(f_{s0}\) and \(I_{s0}\) to be the \emph{right modification}
of \(f_{s}\) via \(I_{s0} \coloneqq  I_{s} \left[ - \frac{1}{2}, 0 \right]\); and
$f_{s0}\coloneqq f_{s}(x)$ if \(x \notin I_{s} \left[ - \frac{1}{2}, 1\right]\), and
if \(x \in I_{s} \left[ - \frac{1}{2}, 1\right]\)
\begin{equation*}
f_{s0}(x) \coloneqq b_{\abs{s} + 1} - 2^{- 3 (\abs{s} + 1)} + 2^{-\abs{s} - 1}
    \abs*{x - I_{s} \left(- \frac{1}{4} \right)}.
\end{equation*}

%\begin{align*}
%  I_{s0} &\coloneqq
%  I_{s} \left[ - \frac{1}{2}, 0 \right]
%  \\
%  f_{s0}(x) &\coloneqq
%  \begin{dcases}
%    b_{\abs{s} + 1} - 2^{- 3 (\abs{s} + 1)} + 2^{-\abs{s} - 1}
%    \abs*{x - I_{s} \left(- \frac{1}{4} \right)}, &
%    \text{if } x \in I_{s} \left[
%      - \frac{1}{2}, 1
%    \right]\\
%    f_{s}(x), & \text{otherwise.}
%  \end{dcases}
%\end{align*}

\begin{figure}[h]
  \centering
  \small
  \begin{tikzpicture}[x=2.5em,y=2.5em]
    %% f_{s}
    \draw[thick] (0,6) coordinate(top left)
    [namex ={\(-1\)}]
    -- ++(4,-4) %($4*(1,-1)$)
    -- ++(4,+4) %($4*(1,+1)$)
    [namex = {\(+1\)}]
    ;
    %% f_{s0}
    \draw[very thick] (0,6)
    -- ++(2,-2) coordinate(break) %($2*(1,-1)$)
    [namex = {\(- \frac{1}{2}\)}]
    -- ++(1,-1/2)
    [namex = {\(- \frac{1}{4}\)}]
    -- ++(5,5/2) %($5*(1,+1/2)$)
    ;
    \draw[dashed] (break) -- ++(right:6)
    node[right]{\(b_{\abs{s} + 1}\)};
    \draw[dashed] (top left) -- ++(right:8)
    node[right]{\(b_{\abs{s}}\)};
    %% baseline
    \draw[thin] (0,0) node[left]{\(I_{s}\)} -- (8,0);
    %% missing midpoint
    \path (4,4) [namex = {\(0\)}];
  \end{tikzpicture}\\
  %\quad
  \hspace{-0.9cm}
  \begin{tikzpicture}[x=2.5em,y=2.5em]
    %% f_{s0}
    \draw[very thick] (0,6)
    [namex ={\(-1\)}]
    -- ++(2,-2) %($2*(1,-1)$)
    [namex = {\(- \frac{1}{2}\)}]
    -- ++(1,-1/2)
    [namex = {\(- \frac{1}{4}\)}]
    -- ++(5,5/2) %($5*(1,+1/2)$)
    [namex = {\(1\)}]
    ;
    %% f_{s1}
    \draw[dotted, thick] (8,6)
    -- ++(-2,-2) %($2*(-1,-1)$)
    [namex = {\(\frac{1}{2}\)}]
    -- ++(-1,-1/2) coordinate(min)
    [namex = {\(\frac{1}{4}\)}]
    -- ++(-5,5/2) %($5*(-1,+1/2)$)
    ;
    %% baseline
    \draw[thin] node[left]{\(I_{s}\)} (0,0) -- (8,0);
    %% missing midpoint
    \path (4,4) [namex = {\(0\)}];
  \end{tikzpicture}
  \caption{Above: right modification;
    the solid normal line is before the modification,
    the solid thick line after it.
    Below: right modification is the solid thick line;
    left modification is the dotted line.}
  \label{fig:modifications}
\end{figure}
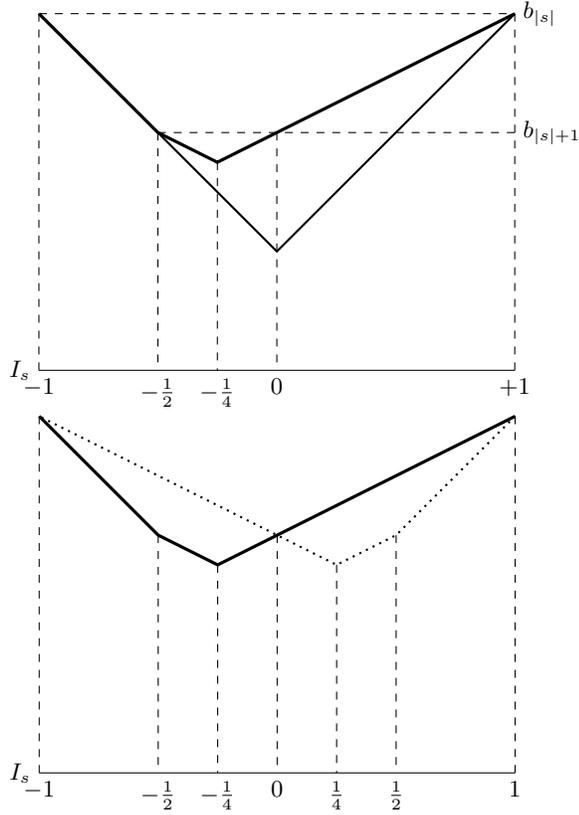

Similarly, the \emph{left modification} \(f_{s1}\) of \(f_s\)
is the reflection of \(f_{s0}\) with respect to \(I_{s}(0)\),
and \(I_{s1}\) is the reflection of \(I_{s0}\)
with respect to \(I_{s}(0)\). %,
%i.e.,
%\begin{align*}
%  I_{s1} &\coloneqq
%  I_{s} \left[ 0, \frac{1}{2} \right]
%  \\
%  f_{s1}(x) &\coloneqq
%  \begin{dcases}
%    b_{\abs{s} + 1} - 2^{- (3 \abs{s} + 1)} + 2^{- \abs{s} - 1}
%    \abs*{x - I_{s} \left( \frac{1}{4} \right)}, &
%    \text{if } x \in I_{s} \left[
%      -1, \frac{1}{2}
%    \right]\\
%    f_{s}(x), & \text{otherwise.}
%  \end{dcases}
%\end{align*}
Observe that \(I_{s0}, I_{s1} \subseteq I_{s}\)
and \(\interior{I_{s0}} \cap \interior{I_{s1}} = \emptyset\).

This finishes the definition of the \(f_{s}\).
Clearly, these functions are convex and Lipschitz-continuous
with Lipschitz constant \(1\), satisfying
\ref{item:intLength}--\ref{item:bounds}.

We establish the packing property for \(\mathcal{F}\). 

\begin{lemma}
  \label{lem:one-dim-packing}
  The family \(\mathcal{F}\) satisfies the packing property
  for \(M = \lfloor(1/3)\log(1/\varepsilon)\rfloor\).
\begin{proof}
Note that \(f_{S}\) has its minimum at the midpoint of \(I_{S}\),
and the function value at the endpoints of \(I_{S}\) are
at least \((1/2)^{3M} \geq \varepsilon\) larger than the
value at the midpoint.
Therefore every \(\varepsilon\)-optimal solution lies in
the interior of \(I_{S}\),
i.e.,
\(\mathcal{S}_{\varepsilon}(f_{S}) \subseteq \interior{I_{S}}\).
Therefore by \ref{item:disjoint-if-not-prefix},
the \(\mathcal{S}_{\varepsilon}(f_{S})\) are pairwise disjoint.
\end{proof}
\end{lemma}

In the following \(F \in \mathcal F\) will be an instance picked
uniformly at random. The random variable \(S\) will be the associated
string of length \(M\) so that \(F = f_S\) and \(S\) is also distributed uniformly.

\subsubsection*{Reduction to the String Guessing Problem}
\label{sec:char-oracle-answ}

We will now provide an oracle for family \(\mathcal{F}\)
that can be emulated by the String Guessing Oracle.
As a first step, we relate the query point \(x\)
with the indexing strings of the functions. At a high level,
the lemma below shows the existence of a prefix of the unknown
string determining most of the local behavior of the function 
at a 
given query point. From this we will prove in  
Lemma~\ref{lem:learn-nonsmooth-1-dim}
that the oracle answer only reveals this prefix.

\begin{lemma}
  \label{lem:answer-nonsmooth-1-dim}
  Let \(x \in [-1, +1]\) be a query point.
  Then there is a non-empty binary string \(s\) with \(l \coloneqq
  \card{s} \leq M\) with the following properties.
  \begin{enumerate}
  \item\label{item:string-decreasing}
    \(f_{\flipCut{1}{s}}(x) \geq b_{1} > f_{\flipCut{2}{s}}(x) \geq
    \dotsb \geq b_{l-1} > f_{\flipCut{l}{s}}(x) \geq f_{s}(x)\). If \(l < M\) then also \(f_{s}(x) \geq b_{l}\).
  \item\label{item:prefix-gives-value}
    Every binary string \(t\) of length \(M\)
    has a unique prefix \(p\)
    from \(\{\flipCut{1}{s}, \dotsc, \flipCut{l}{s}, s\}\).
    Moreover,
    \(f_{t}(x) = f_{p}(x)\).
  \end{enumerate}

\begin{proof}
Let \(s_{0}\) be the longest binary string
of length less than \(M\),
such that \(x\) lies in the interior of \(I_{s_{0}}\).
We choose \(s\) to be the one of
the two extensions of \(s_{0}\) by \(1\) bit,
for which \(f_{s}\) has the smaller function value at \(x\)
(if the two values are equal, then either extension will do).
Let \(l \coloneqq \card{s}\),
thus \(f_{\flipCut{l}{s}}(x) \geq f_{s}(x)\).

Note that by the choice of \(s_{0}\),
the point \(x\) is not an interior point of \(I_{s}\)
unless \(l = M\).
By \ref{item:disjoint-if-not-prefix},
the point \(x\) is neither an interior point of any of the
\(I_{\flipCut{1}{s}}, \dotsc, I_{\flipCut{l}{s}}\).

To prove \ref{item:prefix-gives-value},
let $t$ be any binary string of length $M$.
The existence and uniqueness of a prefix \(p\) of \(t\)
from the set \(\{\flipCut{1}{s}, \dotsc, \flipCut{l}{s}, s\}\)
is clear.
In particular,
unless \(p = t = s\) and \(l = M\),
the point \(x\) is not an interior point of \(I_{p}\),
hence \(f_{t}(x) = f_{p}(x)\) follows
from \ref{item:family-increasing}.
When \(p = t\), then \(f_{t}(x) = f_{p}(x)\) is obviously true.

Now we prove \ref{item:string-decreasing}.
Recall that \(f_{\flipCut{l}{s}}(x) \geq f_{s}(x)\)
by the choice of \(s\).
First, if
\(l < M\) then \(x \notin \interior{I_{s}}\) by choice,
hence \(f_{s}(x) \geq b_{l}\) by \ref{item:bounds}.
Second, let us prove  \(f_{\flipCut{i}{s}}(x) \geq b_{i} > f_{\flipCut{i+1}{s}}(x)\)
for all \(i \leq l\).
As \(x \notin \interior{I_{\flipCut{i}{s}}}\),
by \ref{item:bounds} we have \(f_{\flipCut{i}{s}}(x) \geq b_{i}\). 
Finally, since \(x \in \interior{I_{\prefix{i}{s}}}\) and 
\( \prefix{i}{s} \sqsubseteq \flipCut{i+1}{s} \),
again by  \ref{item:bounds}
we get \( f_{\flipCut{i+1}{s}}(x) < b_i\).
\end{proof}
\end{lemma}

Our construction of instances encodes prefixes in
level sets of the instance. The previous 
lemma indicates that algorithms in this case need 
to identify a random string, where the oracle reveals 
prefixes of such string. The following lemma
formally shows an emulation by the String Guessing 
Oracle.

\begin{lemma}
  \label{lem:learn-nonsmooth-1-dim}
  There is a single-coordinate local oracle 
  \(\tilde{\mathcal{O}}\) for the family \(\mathcal{F}\) above,
  which is emulated by the
  String Guessing Oracle \(\mathcal{O}_{S}\)
  on strings of length \(M\).
\begin{proof}
We define the emulation functions first, as they determine the
emulated oracle \(\tilde{\mathcal{O}}\).
Let \(x \in [-1,1]\) and \(s\) the string from
Lemma~\ref{lem:answer-nonsmooth-1-dim}.
We define the query emulation function as
\(q(x) \coloneqq (s, \operatorname{id})\).
Moreover, let \(l = \card{s}\).

Now we need to emulate the oracle answer. From Lemma 
\ref{lem:answer-nonsmooth-1-dim} \ref{item:prefix-gives-value} 
there exists a prefix \(P\) of \(S\) such that
\(f_S(x)=f_P(x)\). We define the following function \(p\) of the
\(\mathcal{O}_{S}\) oracle answer
\begin{align*}
  p(x, \texttt{EQUAL}) &\coloneqq s, 
  \\
  p(x, k) &\coloneqq \flipCut{k}{s}.
\end{align*}
Note that \(P = p(x, \mathcal{O}_{S}(q(x))\). 
We claim that  \(p\) depends on \(f_S\) only locally around \(x\). First,
if \(f_s(x)<f_{\flipCut{l}{s}}(x)\) then by  Lemma~\ref{lem:answer-nonsmooth-1-dim}
\ref{item:string-decreasing} \(f_S(x)\) determines \(P\) (and thus \(p\)). Otherwise, 
depending on whether \(f_S\) is increasing or decreasing around \(x\), we
can determine if $P_l=s_l$.

Since \(f_{S}(x) = f_{P}(x)\) and \(f_{S} \geq f_{P}\),
a valid oracle answer to the query point \(x\)
is \(f_{P}(x)\) as function value and
a subgradient \(\nabla f_{P}(x)\) of \(f_{P}\) at \(x\)
as \(\nabla f_{S}(x)\).
Therefore we define the
answer emulation as
\(  a(x, R) \coloneqq (f_{p(x, R)}(x), \nabla f_{p(x, R)}(x)). \)
This provides a single-coordinate local oracle \(\tilde{\mathcal{O}}\)
for the family \(\mathcal{F}\) (the single-coordinate condition
is trivially satisfied when \(n=1\))
that can be emulated by
the String Guessing Oracle \(\mathcal{O}_{S}\).
\end{proof}
\end{lemma}

The previous lemma together with
Lemma~\ref{lem:emulator-complexity} leads to
a straightforward proof of
Theorem~\ref{th:box-nonsmooth-lower-bound} in
the one dimensional case.

\begin{proof}[Proof of Theorem~\ref{th:box-nonsmooth-lower-bound}
  for \(n = 1\)]
Let \(A\) be a black box optimization
algorithm for \(\mathcal F\) accessing
the oracle \(\tilde{\mathcal{O}}\).
As \(\mathcal F\) satisfies the
packing property by Lemma~\ref{lem:one-dim-packing}, in order to find
an \(\varepsilon\)-minimum the algorithm \(A\) has to identify the
string \(s\) defining the function \(f = f_s\) (and
from an \(\varepsilon\)-minimum the string \(s\) can be
recovered). 

Let \(F = f_S\) be the random instance chosen with uniform
distribution.
Together with the emulation defined in
Lemma~\ref{lem:learn-nonsmooth-1-dim}, algorithm \(A\) solves the
String Guessing Problem for strings of length \(M\),
hence requiring at least
\([(1 - P_{e}) M - 1] / 2\)
queries in expectation
with error probability at most \(P_{e}\)
by Proposition~\ref{prop:string-via-oracle}.
Moreover, with probability \(1 - P_{e} - o(1)\),
the number of queries is at least \(\Omega(M)\).
This proves the theorem for \(n=1\) by the choice of \(M\).
\end{proof}

\subsection{Multidimensional case}
\label{sec:mult-case}

\subsubsection*{Construction of function family}

In the general \(n\)-dimensional case
the main difference is using a larger indexing string.
Therefore
we choose \(M = \lfloor (1/3)\log (1 / \varepsilon) \rfloor\),
and consider \(n\)-tuples \(s_{1}, \dotsc, s_{n}\) of binary strings
of length \(M\) as indexing set for the function family
\(\mathcal{F}\),
and define the member functions via
\begin{equation} \label{box_family}
  f_{s_{1}, \dotsc, s_{n}}(x_{1}, \dotsc, x_{n}) \coloneqq
  \max_{i \in [n]} f_{s_{i}}(x_{i}),
\end{equation}
where the \(f_{s_{i}}\) are the functions
from the one dimensional case. This way, the size of
$\mathcal{F}$ is $2^{nM}$.
Note that as the \(f_{s_{i}}\) are \(1\)-Lipschitz,
the \(f_{s_{1}, \dotsc, s_{n}}\) are \(1\)-Lipschitz
in the \(L^{\infty}\) norm, too. We prove that $\mathcal{F}$
satisfies the packing property.

\begin{lemma} \label{pack_prop_box}
  The family \(\mathcal{F}\) above satisfies the packing property for
  \(M = \lfloor(1/3) \log (1 / \varepsilon) \rfloor\).
\begin{proof}
As the minimum values of all the one dimensional \(f_{s_{i}}\)
coincide,
obviously the set of \(\varepsilon\)-minima
of \(f_{s_{1}, \dots, s_{n}}\) is the product of its
components:
\begin{equation*}
  \mathcal{S}_{\varepsilon}(f_{s_{1}, \dots, s_{n}})
  = \prod_{i \in [n]} \mathcal{S}_{\varepsilon}(f_{s_{i}}).
\end{equation*}
Hence the claim reduces to the one dimensional case,
proved in Lemma~\ref{lem:one-dim-packing}.
\end{proof}
\end{lemma}

Let \(S = (S_{1}, \dotsc, S_{n})\) denote the tuple of strings
indexing the actual instance,
hence the \(S_{i}\) are mutually independent uniform 
binary strings; and let \( F = f_{S_1,\ldots,S_n}. \)

\subsubsection*{Reduction to the String Guessing Problem}

We argue as in the one dimensional case,
but now the string for the String Guessing Oracle
is the concatenation of the strings \(S_{1}, \dotsc, S_{n}\),
and therefore has length \(n M\).
\begin{lemma}
  \label{lem:learn-nonsmooth-adversarial}
  There is a single-coordinate oracle \(\tilde{\mathcal{O}}\)
  for family \(\mathcal{F}\) that can be emulated by
  the String Guessing Oracle \(\mathcal{O}_{S}\)
  where \(S\) is the concatenation of the
  \(S_{1}, \dots, S_{n}\).
\end{lemma}

Before proving the result, let us motivate our choice
for the first-order oracle. The general case arises 
from an interleaving of the case \(n = 1\). 
As we have seen in the proof of 
Lemma~\ref{lem:learn-nonsmooth-1-dim}, 
for \(n=1\) querying the first-order oracle 
leads to querying prefixes. By \ref{item:family-increasing}, 
for any
prefix \(S'\) of \(S\) we have \(f_{S'} \leq f_S\);
this gives a lower bound on the unknown instance.
By querying a point $x$ we obtain such a prefix with the 
additional property \(f_{S'}(x)=f_S(x)\), which localizes
the minimizer in an interval, and thus provides an upper
bound on its value.

Now, for general \(n\) we want to upper bound the maximum
by prefixes of the hidden strings. In particular, there is no
use to querying any potential prefixes \(u\) for coordinate \(i\) 
such that \(f_u(x_{i})\) is strictly smaller than the candidate 
maximum; they are not revealed by the oracle.

The query string for the String Guessing Oracle now arises by 
interleaving the query 
strings for each coordinate. In particular, if we restrict the query 
string to the substring consisting only of prefixes for a specific 
coordinate \(i\), then these substrings should be ordered by 
\(\sqsubseteq\), which is precisely the ordering we used for 
the case \(n = 1\) as a necessary condition. Thus, a natural 
way of interleaving these query strings is by their objective 
function value. Moreover, refining this order by the lexicographic order
on coordinates will induce a single-coordinate oracle.

\begin{proof}
Let \(x=(x_1,\ldots,x_n)\) be a query point. For a family of strings \(\{S_i\}_i\)
let \(S\) be their concatenation, and for notational convenience
let \(S_{i,h}\) denote the \(h\)-th bit of \(S_i\). Applying
Lemma~\ref{lem:answer-nonsmooth-1-dim} to each coordinate 
\(i\in[n]\), there is a number \(l_i\) and a string \(s_i\) of length 
\(l_i\) associated to the point \(x_i\). 

We define the confidence order \(\prec\) of labels \((i, h)\) with 
\(i \in [n]\) and \(h \in [l_{i}]\) as the one induced by the lexicographic 
order on the pairs \((- f_{\flipCut{h}{s_{i}}}(x_{i}), i)\) i.e.,
\begin{multline} \label{lexic-order}
  (i_{1}, h_{1}) \prec (i_{2}, h_{2}) \\
  \iff
  \begin{cases*}
    f_{\flipCut{h_{1}}{s_{i_{1}}}}(x_{i_{1}})
  > f_{\flipCut{h_{2}}{s_{i_{2}}}}(x_{i_{2}}) \qquad\quad\mbox{or} \\
    f_{\flipCut{h_{1}}{s_{i_{1}}}}(x_{i_{1}})
  = f_{\flipCut{h_{2}}{s_{i_{2}}}}(x_{i_{2}})\ \wedge \ i_{1} \leq i_{2}.
  \end{cases*}
\end{multline}
We restrict to the labels \((i,h)\) with
\(f_{\flipCut{h}{s_{i}}}(x_{i}) \geq \max_{j \in [n]} f_{s_{j}}(x_{j})\) 
(there is no use to query the rest of labels, as pointed out
above). Let \((i_{1}, h_{1}), \dotsc, (i_{k}, h_{k})\)
be the sequence of these labels in
\(\prec\)-increasing order.
Let \(t\) be the string of length \(k\)
with \(t_{m} = s_{i_{m}, h_{m}}\) for all \(m \in [k]\).
We define the query emulation as
\(q(x) = (t, \sigma)\)
with \(\sigma_{m} \coloneqq (i_{m}, h_{m})\).

We define a coordinate \(j\) and \(p\) a prefix of \(S_{j}\))
as helper functions in \(x\) and the answer of the String Guessing Oracle.
for the answer emulation \(a\)
(with the intent of having \(f_{S}(x) = f_{p}(x_{j})\).
If the oracle answer is \texttt{EQUAL},
then we choose \(j=i_k\),
and set \(p \coloneqq \prefix{h_k}{s_j}\).
If the oracle answer is a number \(m\)
then we set \(j \coloneqq i_{m}\) and
\(p \coloneqq \flipCut{h_{m}}{s_{i_{m}}}\).

Analogously as in the proof of Lemma~\ref{lem:learn-nonsmooth-1-dim},
both \(p\) and \(j\) depend only on \(x\) and on the local behavior
of \(f_S\) around \(x\). Moreover, it is easy to
see that \(f_{S}(x) = f_{p}(x_{j})\) and 
\(f_{S}(y) \geq f_{S_{j}}(y_{j}) \geq f_{p}(y_{j})\) for all \(y\),
which means that \( \nabla f_{p}(x_{j}) e_j\) 
is a subgradient of \(f_S\) at \(x\).

We now define the answer emulation
\[a(x, R) = (f_{p(x, R)}(x_{j(x, R)}), \nabla f_{p(x, R)}(x_{j(x, R)}) e_{j(x,R)}),\]
and thus the oracle
\(\tilde{\mathcal{O}}(x) = a(x, \mathcal{O}_{S}(q(x)))\)
is a first-order local oracle for the family \(\mathcal{F}\) that
can be emulated by  the String Guessing Oracle. Finally, the 
single-coordinate condition is satisfied from the confidence order 
of the queries, which proves our result.
\end{proof}

We are ready to prove Theorem~\ref{th:box-nonsmooth-lower-bound}.

\begin{proof}[Proof of Theorem~\ref{th:box-nonsmooth-lower-bound}]  
The proof is analogous to the case \(n = 1\). However, by the emulation
via Lemma~\ref{lem:learn-nonsmooth-adversarial} we solve the
String Guessing Problem for strings of length \(nM\).
Thus by Proposition~\ref{prop:string-via-oracle}
we obtain the claimed bounds the same way as in the case \(n=1\).
\end{proof}

\section{Single-coordinate oracle complexity for \(L^{p}\)-Balls}
\label{sec:large-scale-compl}

In this section we examine the complexity of
convex nonsmooth optimization on
the unit ball $B_{p}(0,1)$ in the \(L^{p}\) norm
for \(1 \leq p < \infty\). Again, we restrict our analysis
to the case of single-coordinate oracles.
We distinguish
the large-scale case (i.e., $\varepsilon \geq 1 / n^{\max\{p,2\}}$),
and low-scale case (i.e., $\varepsilon \leq n^{-1/\max\{p,2\}-\delta}$, for
fixed $\delta>0$).

\subsection{Large-scale case}
\label{sec:large-scale-case}

\begin{theorem}
\label{thm:large-scale-ball}
  Let $1 \leq p<\infty$ and \(\varepsilon \geq 1 / \sqrt[p]{n}\).
  There exists a finite family \(\mathcal{F}\)
  of convex Lipschitz-continuous functions
  in the \(L^{p}\) norm
  with Lipschitz constant \(1\)
  on the \(n\)-dimensional unit ball \(B_{p}(0,1)\), and 
  a single-coordinate local oracle \(\tilde{\mathcal{O}}\)
  for \(\mathcal{F}\),
  such that
  both
  the distributional
  and
  the high-probability oracle complexity
  of finding an \(\varepsilon\)-minimum
  under the uniform distribution
  are
  \(\Omega\left( 1 / \varepsilon^{\max\{p, 2\}} \right)\).
  
 For bounded-error algorithms with error probability at most \(P_e\),
 the distributional complexity is
 \(  \Omega \left( (1-P_e)/{\varepsilon^{\max\{p,2\}}} \right) \),
 while the high probability complexity is
 \(\Omega\left( 1/\varepsilon^{\max\{p,2\}} \right)\).
\end{theorem}

\begin{remark}[The case $p=1$]
  For $p=1$, the lower bound can be improved to
  $\Omega\left(\frac{\ln n}{\varepsilon^{2}}\right)$
  by a nice probabilistic argument,
  see \cite[Section~4.4.5.2]{Nemirovski:1983}.
\end{remark}

As in the previous section, we will construct a single-coordinate 
oracle that can be emulated by the String Guessing Oracle.
As the lower bound does not depend on the dimension,
we shall restrict our attention to the first
\(M = \Omega (1 / \varepsilon^{\max\{p, 2\}})\) coordinates.
For these coordinates,
it will be convenient to work in an orthogonal basis
of vectors with maximal ratio of \(L^{p}\) norm and
\(L^{2}\) norm,
to efficiently pack functions in the \(L^{p}\)-ball.
For \(p \geq 2\) the standard basis vectors \(e_{i}\)
already have maximal ratio,
but for \(p < 2\) it requires a basis of vectors with
all coordinates of all vectors being \(\pm 1\),
see Figure~\ref{fig:max-l-p}.
In particular, in our working basis the \(L^{p}\) norm
might look different than in the standard basis.
\begin{figure}[htb]
  \centering
  \begin{tikzpicture}[x=1cm, y=1cm]
    % Euclidean ball
    \draw[gray] (0,0) circle[radius=1];
    \node[anchor=north] at (0,-1) {\(p > 2\)};
	% coordinate vectors
	\draw (0,0)
	\foreach \x/\y in {0/1,1/0, 0/-1, -1/0}
	  {edge[-latex] (\x,\y)};
    % l3-ball
    \draw
	plot[id=up3, domain=-1:1, samples=53, smooth]
	function{(1-((abs(x))**3))**(1.0/3)}
	plot[id=down3, domain=-1:1, samples=53, smooth]
	function{-(1-((abs(x))**3))**(1.0/3)}
	;
  \end{tikzpicture}
  \hfil
  \begin{tikzpicture}
    % Euclidean ball
    \draw[gray] (0,0) circle[radius=1];
    \node[anchor=north] at (0,-1) {\(p < 2\)};
    % l4/3-ball
    \draw
	plot[id=up34, domain=-1:1, samples=53, smooth]
	(\x^3, {sqrt(sqrt((1 - abs(\x)^4)^3))})
	%function{(1-(abs(x))**(4.0/3))**(3.0/4)}
	plot[id=down34, domain=-1:1, samples=53, smooth]
	(\x^3, {-sqrt(sqrt((1 - abs(\x)^4)^3))})
	%plot[domain=-1:1, samples=51, smooth]
	%function{-(1-((abs(x))**3))**(1.0/3)}
	;
	% ±1 vectors
	\draw[scale=1/sqrt(2)] (0,0)
	\foreach \x/\y in {1/1, 1/-1, -1/1, -1/-1}
	  {edge[-latex] (\x,\y)};
  \end{tikzpicture}
  \caption{\label{fig:max-l-p}%
    Unit vectors of maximal \(L^{p}\) norm
    together with the unit Euclidean ball in gray and
    the unit \(L^{p}\)-ball in black.}
\end{figure}
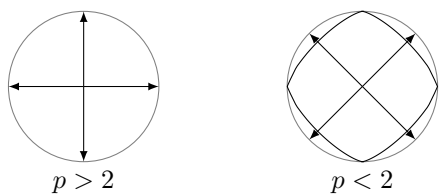
We shall present the two cases uniformly,
keeping the differences to a bare minimum.

The exact setup is as follows.
Let \(\norm[p]{\cdot}\) denote the \(L^{p}\) norm in the original
basis.
Let \(r \coloneqq \max\{p, 2\}\) for simplicity.
We define \(M\) and the working basis
for the first \(M\) coordinates,
such that the coordinates as functions will have Lipschitz constant
at most \(1\).

\begin{casesblock}
\begin{case}{\(2\leq p <\infty\).}
We let
\(M \coloneqq \left\lfloor\frac{1}{\varepsilon^p} \right\rfloor - 1\).
The working basis is chosen to be the standard basis.
\end{case}
\begin{case}{\(1\leq p<2\).}
Let \(l\) be the largest integer with $1/\varepsilon^{2} > 2^{l}$,
and  define \(M\coloneqq 2^l\).
Since \(\varepsilon\geq 1/n^2\),
obviously \(M < 1 / \varepsilon^{2} \leq n\).
In the standard basis,
the space \(\RR^{2}\) has an orthogonal basis of \(\pm 1\) vectors,
e.g., \((1,1)\) and \((1,-1)\).
Taking \(l\)-fold tensor power,
we obtain an orthogonal basis of \(\RR^{M}\)
consisting of \(\pm 1\) vectors \(\nu_{i}\) in the standard basis.
We shall work in the orthonormal basis
\(\xi_{i} \coloneqq  \nu_{i} / \sqrt{M}\).
Note that
the coordinate functions \(\sprod{\xi_{i}}{\cdot}\) of this basis
have Lipschitz constant
at most \(1\) with respect to \(\norm[p]{\cdot}\),
as
\(\sprod{\xi_{i}}{x} \leq \norm[q]{\xi_{i}} \norm[p]{x}\)
for all \(x\),
and \(\norm[q]{\xi_{i}} = \sqrt[q]{M} / \sqrt{M} < 1\),
where \(q\) is chosen such that \(1/p + 1/q = 1\).
\end{case}
\end{casesblock}

Clearly in both cases,
\(M \leq n\) and \(M = \Omega (1 / \varepsilon^{r})\),
but \(M < 1 / \varepsilon^{r}\).
Note that \(\norm[p]{x} \leq \tnorm{x}\) if \(p < 2\).
As the working basis is orthonormal,
\(\tnorm{\cdot}\) is the \(L^{2}\) norm
in both the original basis and the working basis.

\subsubsection*{Construction of function family}

We define our functions \(f_{s} \colon B_{p}(0,1)\to \RR\)
as maximum of (linear) coordinate functions:
\begin{equation}\label{eq:large_scale_instance}
f_s(x) = \max_{i\in[M]} s_i x_i,
\end{equation}
where the \(x_i\) are
the coordinates of \(x\) in our working basis.

We parameterize the family
\(\mathcal{F} = \{f_{s} : s \in \{-1, +1\}^{M}\}\)
via sequences \(s = (s_{1}, \dots, s_{M})\) of signs \(\pm 1\) of
length \(M\). By the above this family satisfies the requirements of
Theorem~\ref{thm:large-scale-ball}. We establish the
packing property for \(\mathcal{F}\).

\begin{lemma} \label{large_scale_pack_prop}
  The family \(\mathcal{F}\) satisfies the packing property.
\begin{proof}
Let \(x = (x_{1}, \dots, x_{n})\) be
an \(\varepsilon\)-minimum of \(f_{s}\).
We compare it with
\[
x^{*} \coloneqq \left(
  - \frac{s_{1}}{\sqrt[r]{M}}, \dots,
  - \frac{s_{M}}{\sqrt[r]{M}},
  0, \dots, 0 \right).
\]
Recall that \(r = \max\{p,2\}\).
The vector \(x^{*}\) lies in the unit \(L^{p}\)-ball,
as \(\norm[p]{x^{*}} = M^{1/p - 1/r} \leq 1\).

Therefore, as \(M < 1 / \varepsilon^{r}\),
we obtain for all \(i \in [M]\)
\begin{equation*}
\begin{multlined}
  s_i x_i \leq f_s(x_{1}, \dots, x_{n})
  \leq f_s^{\ast} + \varepsilon \\ \quad
  \leq f_s(x^{*}) + \varepsilon 
  = - \frac{1}{\sqrt[r]{M}} + \varepsilon < 0,
\end{multlined}
\end{equation*}
i.e., \(s_{i} = - \sign x_{i}\).
Hence every $\varepsilon$-minimum \(x\) uniquely
determines $s$, proving the packing property.
\end{proof}
\end{lemma}

Let
\(F \in \mathcal F\)
be chosen
uniformly at random, and let \(S\) be the associated
string of length \(M\) so that \(F = f_S\) and thus \(S\in\{-1, +1\}^{M}\) is 
uniformly distributed.

\subsubsection*{Reduction to the String Guessing Problem}

The main idea is that the algorithm learns solely
some entries \(S_{i}\) of the string \(S\) from an oracle answer.

\begin{lemma}
  \label{lem:learn-nonsmooth-large-scale}
  There is a single-coordinate local oracle \(\tilde{\mathcal{O}}\)
  that can be emulated by the String Guessing Oracle \(\mathcal{O}_{S}\).
\begin{proof}
To better suit the present problem,
we now use \(\pm 1\) for the values of bits of strings.

Given a query \(x\),
we introduce an ordering \(\prec\)
on the set of coordinates \(\{1, 2, \dots, M\}\):
we map each coordinate \(i\)
to the pair \((- \abs{x_{i}}, i)\),
and take the lexicographic order on these pairs, i.e.,
\begin{equation*}
  i_{1} \prec i_{2} \iff
  \begin{cases*}
    \abs*{x_{i_{1}}} > \abs*{x_{i_{2}}} & or \\
    \abs*{x_{i_{1}}} = \abs*{x_{i_{2}}}\ \wedge \ i_{1} \leq i_{2}.
  \end{cases*}
\end{equation*}
Let \(\sigma(1), \dots, \sigma(k)\) be the indices \(i \in [M]\)
put into \(\prec\)-increasing order with \(k\) the minimum 
between \(M\) and the \(\prec\)-first $i$ s.t. \(x_{i} = 0\).
Let \(s\) be the string of length \(k\)
with \(s_{j} = - \sign x_{\sigma(j)}\).
If \(x_{\sigma(j)} = 0\), we put \(s_{j} = +1\).
(The value \(-1\) would also do.)
The query emulation \(q\) is defined via
\(q(x) \coloneqq (s, \sigma)\).

We now define helper functions \(J\) and \(p\)
in \(x\) and an answer of \(\mathcal{O}_{S}\).
We set 
\begin{align*}
  J(x, \texttt{EQUAL}) &\coloneqq k,
  &
  p(x, \texttt{EQUAL}) &\coloneqq s_{k},\\
  %&
  J(x, j) &\coloneqq j,
  &
  p(x, j) &\coloneqq - s_{j}.
\end{align*}

For the remainder of the proof we drop the arguments of these
functions and simply write \(J\) and \(p\)
instead of
\(J(x, \mathcal{O}_{S}(q(x)))\) and
\(p(x, \mathcal{O}_{S}(q(x)))\), respectively to ease readability.

Actually, \(J\)
is the \(\prec\)-smallest index \(j\) with \(f_{S}(x) = S_{j} x_{j}\)
and if \(x_{j} = 0\) and \(j < \sigma(k)\)
then furthermore \(- e_{j} \in \partial f(x)\).
Note that \(p = S_{J}\) and
therefore \(p e_{J}\) is a subgradient of 
\(f_S\) at \(x\) (and also at \(x - t \sign(x) e_{j}\) for  small \(t
\geq 0\)).
These are local conditions uniquely determining \(J\) and \(p\).
Hence, \(J\) and \(p\) are local.

We define the query emulation \(a\) via
\(a(x, R) \coloneqq (p(x, R) x_{J(x, R)}, p(x, R) e_{J(x, R)})\).
Oracle \(\tilde{\mathcal{O}}\) is defined by the emulation
\( \tilde{\mathcal{O}}(x) = a(x, \mathcal{O}_{S}(q(x)))\),
which is clearly single-coordinate and local.
Thus \(\tilde{\mathcal{O}}(x)\) is a valid answer to query \(x\).
\end{proof}
\end{lemma}

We are ready to prove Theorem~\ref{thm:large-scale-ball}

\begin{proof}[Proof of Theorem~\ref{thm:large-scale-ball}]
The proof is analogous to the proof of
Theorem~\ref{th:box-nonsmooth-lower-bound}.
 Given the oracle \(\mathcal{O}\) in
Lemma~\ref{lem:learn-nonsmooth-large-scale}, every black box
algorithm \(A\) having access to this oracle solves
the String Guessing Problem for strings of
length \(M = \Theta(1/\varepsilon^{\max\{p,2\}})\) 
using the String Guessing Oracle only. Hence the
claimed lower bounds are obtained by  
Proposition~\ref{prop:string-via-oracle}. 
\end{proof}

\subsection{The low-scale case: reduction to the box case} \label{low-scale}

We show that for small accuracies, the $L^p$-ball lower bound 
follows from Theorem~\ref{th:box-nonsmooth-lower-bound}.
Note that the lower bound below is only optimal for \(p \geq 2\),
but it suffices for the optimal, more general bounds in
Section \ref{sec:oracle_ind_LB}.

\begin{proposition}
  \label{reduction_ball_to_box}
  Let $1\leq p<\infty$,
  and \(\varepsilon \leq n^{-\frac{1}{p} - \delta}\) with \(\delta > 0\).
  There exists a family \(\mathcal{F}\)
  of convex Lipschitz-continuous functions
  in the \(L^{p}\) norm
  with Lipschitz constant \(1\)
  on the \(n\)-dimensional unit Euclidean ball \(B_{p}(0,1)\),
  and a single-coordinate oracle for family \(\mathcal{F}\),
  such that
  both
  the distributional
  and the high-probability oracle complexity
  of finding an \(\varepsilon\)-minimum
  under the uniform distribution
  is
  \(\Omega\left( n \log \frac{1}{\varepsilon} \right)\).

 For algorithms with error probability at most \(P_e\),
 the distributional complexity is
 \(  \Omega \left( (1-P_e)n \log \frac{1}{\varepsilon} \right) \)
 and the high probability complexity is
 \(\Omega\left( n \log \frac{1}{\varepsilon} \right)\).
\begin{proof}
  The proof is based on a rescaling argument. 

%\begin{casesblock}
%\begin{case}{\(2\leq p<\infty\).} 
  We have
  \([-\frac{1}{\sqrt[p]{n}},\frac{1}{\sqrt[p]{n}}]^{n}
  \subseteq B_{p}(0,1)\)
  and thus by Theorem~\ref{th:box-nonsmooth-lower-bound} there exists
  a family of convex Lipschitz-continuous functions
  with Lipschitz constant \(1\)
  (in the $L^{\infty}$ norm, therefore also in the $L^{p}$ norm),
  and a single-coordinate oracle for \(\mathcal{F}\),
  with
  both
  distributional oracle complexity
  and high-probability oracle complexity
  \(\Omega\left( n \log \frac{1}{\varepsilon \sqrt[p]{n}}\right) =
  \Omega\left( n \log  \frac{1}{\varepsilon}\right)\) for large
  \(n\), where the last equality follows from the fact that for
  \(\varepsilon \leq n^{- 1/p - \delta}\) with \(\delta > 0\)
  we have \(\varepsilon \sqrt[p]{n} \leq
  \varepsilon^{\frac{\delta}{1/p + \delta}}\).
%\end{case}
%\begin{case}{\(1\leq p<2\).} 
%\end{case}
%\end{casesblock}

\end{proof}
\end{proposition}

For
the case of the \(L^p\)-ball with \(1 \leq p < \infty\), we
thus close the gap exhibited in Figure~\ref{fig:smallVsLarge} for
arbitrary small but fixed \(\delta > 0\).

\begin{remark}[Understanding the dimensionless speed up in terms of entropy]
  The observed (dimensionless) performance for the \(L^p\)-ball, for $2\leq p<\infty$,  has
  a nice interpretation when comparing the total entropy of the function
  families. Whereas in
  the unit box we could pack up to roughly \(2^{n
    \log \frac{1}{\varepsilon}}\) instances with
  nonintersecting \(\varepsilon\)-solutions, we can only pack
  roughly  \(2^{1/\varepsilon^p}\) into the $L^p$-ball. This drop in
  entropy alone can explain the observed speed up. 

We give some intuition by comparing the volume of the unit box
with the volume of the inscribed unit \(L^{p}\)-ball.
Suppose that there are
\(K_{n} \approx 2^{n \log 1/\varepsilon}\)
\lq{}equidistantly\rq{} packed instances in the box;
this number is roughly the size of the function family used above.
\begin{figure}[htb]
  \centering
  \begin{tikzpicture}[x=.7mm, y=.7mm]
    \draw[very thick] ( 0, 0) rectangle (40,40)
                      (20,20) circle[radius=20];
    \foreach \x in {2,6,...,38}
      \foreach \y in {2,6,...,38}
      {\filldraw (\x,\y) circle[radius=.5];
       \draw     (\x,\y) circle[radius= 2];}
  \end{tikzpicture}
\caption{\label{fig:equidistant-packing}%
  Equidistantly packed points with a neighbourhood
  in a ball and a box.
  The number of points in each is proportional to its volume.}
\end{figure}
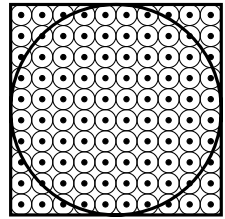
Intersecting with the \(L^{p}\)-ball,
see Figure~\ref{fig:equidistant-packing} for an illustration,
we end up with roughly \(K_{n} V_{n}\) instances,
where \(V_{n} = (2 \Gamma(1/p + 1))^{n} / \Gamma(n/p + 1)\)
is the volume of the unit ball.
For the boundary case \(\varepsilon =1/\sqrt[p]{n}\):
\begin{equation*}
 \begin{split}
  \entropy{F} \approx& \,\log K_{n} V_{n}
  \\
  \approx& \,\,
  n \log n^{1/p} +
  n \left(
    1 + \frac{1}{p} \log \frac{1}{p} - \frac{1}{p}
    + \log \sqrt{2 \pi}
  \right) \\ 
  &-
  \left(
    \frac{n}{p} \log \frac{n}{p} - \frac{n}{p}
    + \log \sqrt{2 \pi}
  \right)
  \\
  \approx&
  \,\,n \left(
    1 + \log \sqrt{2 \pi}
  \right)
  \approx
  \frac{1 + \log \sqrt{2 \pi}}{\varepsilon^{p}},
 \end{split}
\end{equation*}
i.e., the entropy of the function family in the ball drops
significantly, being in line with the existence of fast methdods in
this case.
\end{remark}

\section{Lower bounds for arbitrary local oracles}
\label{sec:oracle_ind_LB}

We extend our results in Sections
\ref{sec:oracle-compl-box}  and \ref{sec:large-scale-compl}
to arbitrary local oracles. The key observation 
is that for query points where the instance is locally linear 
the subdifferential is a singleton, and thus 
any local oracle reduces to the single-coordinate oracle
studied in previous sections.
Thus, we prove lower bounds by perturbing our instances
in such a way that we avoid singular
query points with probability one.

We present full proofs for expected case (distributional) lower bounds, 
however observe that lower bounds w.h.p. (and with bounded error)
follow analogous 
arguments by averaging on conditional probabilities, instead of 
conditional expectations. 

Before going into the explicit constructions, let us present a useful tool
for analyzing arbitrary local oracles. We will show there exists a 
\emph{maximal oracle} \(\mathcal{O}\) such that any local oracle
can be emulated by  \(\mathcal{O}\). This way, it suffices to show
lower bounds on \(\mathcal{O}\) to deduce lower bounds for arbitrary local
oracles.

\begin{definition}[Maximal oracle]
Let $\mathcal{F}$ be a family of real-valued functions over 
a domain $X$.
We define the \emph{maximal oracle} \(\mathcal{O}\) as the one that for
query $x\in X$ provides as answer %$\mathcal{O}_f(x)$, given by
%the family functions $g\in \mathcal{F}$ such that 
%there exists a neighborhood around $x$ (possibly depending 
%on \(g\)) where $f = g$.
\[ \mathcal{O}_f(x)\coloneqq\{g\in \mathcal{F}:\,\, \exists B
\mbox{ neighborhood of $x$ s.t. } f|_B\equiv g|_B\}, \]
where in the expression above the neighborhood $B$ of $x$ 
possibly depends on $g$.
\end{definition}
By definition, \(\mathcal{O}\) is a local oracle. Let us now prove the
claimed property.

%, let \(\mathcal{O}\) be the oracle answering
%for query $x\in X$ the family of functions $g\in \mathcal{F}$ such that
%there exists a neighborhood around $x$ (possibly depending on \(g\))
%where $f = g$. Clearly \(\mathcal{O}\) is a local oracle. 

\begin{lemma}%[Maximal Oracle]
\label{lem:maximalOracle}
Let \(\mathcal{F}\) be a family of real-valued functions. Then the maximal
oracle  \(\mathcal{O}\) is such that any local oracle \(\mathcal{O}^{\prime}\)
can be emulated by \(\mathcal{O}\)
\begin{proof}
Let $\mathcal{O}^{\prime}$ be any local oracle, and $x$ be a query point.
Let the query emulation be the identity, i.e., \(q(x) \coloneqq x\).
Now, for the answer emulation,
by definition, for instances \(f,g\in\mathcal{F}\),
we have \(\mathcal{O}_f(x)=\mathcal{O}_g(x)\)
if and only if $f= g$ around $x$. Therefore the function 
$a(x,\mathcal{O}_f(x)) = \mathcal{O}^{\prime}_f(x)$ is well-defined; this 
defines an oracle emulation of \(\mathcal{O}^{\prime}\) by \(\mathcal{O}\), proving
the result.
\end{proof}
\end{lemma}

For the rest of the section, let \(\tilde{\mathcal{O}}\) be the single-coordinate
oracle studied in previous sections, and let \(\mathcal{O}\) be the maximal
oracle. Note that we state the theorems below for an arbitrary
local oracle \(\mathcal{O}\), but from Lemma~\ref{lem:maximalOracle}
w.l.o.g. we may
choose for the proofs \(\mathcal{O}\) to be the maximal oracle.

\subsection{Large-scale complexity for $L^p$-Balls}
\label{sec:ellpPerturb}

Recall that in Section~\ref{sec:large-scale-case},
different function families were used for the case $1\leq p <2$
and \(2 \leq p <\infty\).
However, the proof below is agnostic to which family is used.

\begin{theorem} \label{perturb-large-scale}
Let $1\leq p <\infty$, $\varepsilon\geq 1/n^{\max\{p,2\}}$ and let
\(\mathcal{O}\) be an arbitrary local oracle
for the family \(\mathcal{F}\)
of convex Lipschitz-continuous functions in the \(L^p\) norm
with Lipschitz constant \(1\)
on the $n$-dimensional unit ball $B_p(0,1)$.
Then both the distributional
and the high-probability oracle complexity
of finding an $\varepsilon$-minimum 
is $\Omega(1/\varepsilon^{\max \{p, 2\}})$.

For bounded-error algorithms with error bound \(P_e\),
the distributional complexity is
\(  \Omega \left( (1-P_e)/{\varepsilon^{\max\{p,2\}}} \right) \),
and the high probability complexity is
\(\Omega\left( 1/\varepsilon^{\max\{p,2\}} \right)\).
\end{theorem}

Before proving the theorem let us introduce the hard function
family, which is a perturbed version of the hard instances in
Section~\ref{sec:large-scale-case}.

\subsubsection*{Construction of function family}

Let $1\leq p<\infty$,  $\varepsilon \geq 1/n^{\max\{p,2\}}$, and 
$X\coloneqq B_p(0,1)$. Let \(M\) and \(f_s\) be defined as in 
the proof of Theorem~\ref{thm:large-scale-ball}, and 
$\bar\delta \coloneqq \varepsilon/(KM)$, where $K>0$ is a constant.
Consider the infinite family
$\mathcal{F} \coloneqq
\left\{ f_{s,\delta}(x): s\in\{-1,+1\}^M, \, \delta\in[0,\bar\delta]^M \right\}$, where
\[ f_{s,\delta}(x) = f_s(x+\delta). \]
Finally, we consider the random variable $F=f_{S,\Delta}$ on
$\mathcal{F}$
where $S\in\{-1,1\}^M$ and $\Delta\in[0,\bar\delta]^M$ are chosen
independently and uniformly at random.

\begin{proof}
The proof requires two steps: 
first, showing that the subfamily
of instances with a fixed perturbation $\delta$
is as hard as the unperturbed one for the single-coordinate 
oracle. Second, by properly averaging over $\delta$
we obtain the expectation lower bound.

\paragraph*{Lower bound for fixed perturbation under oracle $\tilde{\mathcal{O}}$}

Let $\delta\in[0,\bar\delta]^M$ be a fixed vector, and
$\tilde{\mathcal{F}} = \{f_{s,\delta} \colon s \in \{-1,+1\}^M\}$.
Since $f_{s,\delta}(x)=f_s(x+\delta)$, for a fixed perturbation the
subfamily of instances is just a re-centering of the unperturbed ones.
We claim that the complexity of this family under $\tilde{\mathcal{O}}$ 
is lower bounded by $\expectation{T} \geq \frac{M(1-\varepsilon/K)}{2}$.

In fact, consider the ball
$B_p(-\delta,r)$, where $r=1-\varepsilon/K$. Let $x\in B_p(-\delta,r)$,
then
\[ \norm[p]{x} \leq \norm[p]{x + \delta} + M \bar{\delta}
\leq 1 - \varepsilon/K + \varepsilon/K = 1,\]
so $x\in B_p(0,1)$. Therefore, $B_p(-\delta,r)\subseteq B_p(0,1)$, 
and thus the complexity of $\tilde{\mathcal{F}}$ over $B_p(0,1)$ can 
be lower bounded by
the complexity of the same family over $B_p(-\delta,r)$ (optimization
on a subset is easier in terms of oracle complexity). Now observe
that the problem of minimizing $\tilde{\mathcal{F}}$
over $B_p(-\delta,r)$ under $\tilde{\mathcal{O}}$ is equivalent to
the problem studied in Section~\ref{sec:large-scale-case}, 
only with the radius scaled by $r$. This re-scaled problem has the 
same complexity as the original one, only with an extra $r$ factor. 
Thus,
\[ \expectation{T} \geq  \frac{Mr}{2} = \frac{M(1-\varepsilon/K)}{2}
\quad \forall \delta\in[0,\bar\delta]^M.\]

\paragraph*{Lower Bounds for \(\mathcal{F}\) under oracle $\mathcal{O}$}

To conclude our proof, we need to argue that oracle $\mathcal{O}$
does not provide more information than $\tilde{\mathcal{O}}$
with probability \(1\).
Let $A$ be an algorithm and $T$ the number of queries it
requires to determine $S$ (which is a random variable in both $S$
and $\Delta$). 

We will show first that throughout its trajectory $(X^1,\ldots,X^T)$, 
algorithm $A$ queries singular points of \(f_{S,\Delta}\) with 
probability zero. Formally, we have

\begin{lemma}[of unpredictability, large-scale case]
\label{lem:unpredictability-ball-case}
For an \(\mathcal{O}\)-based algorithm solving family \(\mathcal{F}\)
with queries \(X^1,\ldots,X^T\) we define, 
for \(t\geq 0\), the set of maximizer coordinates as
\[ I^t\coloneqq \{ i\in[M]: S_i(X_i^t+\Delta_i) = f_{S,\Delta}(X^t) \}\]
if \(t\leq T\), and \(I^t=\emptyset\) otherwise, and let us consider the 
event \(E\) where the set of maximizers include at most one
new coordinate at each iteration
\[E \coloneqq \left\{
    I^t \subseteq \bigcup_{s<t} I^s \text{ or } \abs*{I^{t}}
  \leq 1, \quad \forall t\leq T
\right\}.\]
Then $\probability{E}=1$.
\begin{proof}
We prove by induction that before every query \(t\geq 1\) the set 
of `unseen' coordinates
\(I_c^t \coloneqq [M]\setminus (\bigcup_{s< t} I^s)\)
is such that perturbations \((\Delta_i)_{i\in I_c^t}\) are absolutely 
continuous. Moreover, from this we
prove simultaneously that
\[\probability[\Pi_{<t}]
  I^{t} \nsubseteq \bigcup_{s< t} I^s \text{ and }
  {\abs*{I^{t}} > 1}
  = 0.\]

We start from the base case \(t=0\), which is evident since the 
distribution on $\Delta$ is uniform. Now, since singular points 
(for all possible realizations of \(S\)) lie in a smaller dimensional 
manifold, then $\abs{I^1}=1$ almost surely. In the inductive step, 
suppose the claim holds up to \(t\) and consider the 
\((t+1)\)-th query. Then
what the transcript provides for coordinates in \(I_c^{t+1}\)
are upper bounds for the perturbations $\Delta_i$
given $S_i$.
In fact, from the \((t+1)\)-th oracle answer all we obtain are $S_j$ and
$\Delta_j$, where $j$ is such that
$f_{S,\Delta}(X^{t+1})=S_j X_j^{t+1} + \Delta_j$; note that
such \(j\) is almost surely unique among \(j\in I_c^t\),
by induction.  For the rest of the
coordinates $i\neq j$ we implicitly know  
\[ S_i X_i^{t+1}+\Delta_i \leq S_j X_j^{t+1} + \Delta_j,\]
i.e., $\Delta_i \leq D_{i,+}$ if $S_i =1$, and $\Delta_i \leq D_{i,-}$ if 
$S_{i} =-1$; where $D_{i,\pm}$ are constants depending on $
(X^{0},\ldots, X^{t+1})$, $S_j$, $\Delta_j$, but not depending
on any of the other unknowns. Thus, at every
iteration we obtain for non-maximizer coordinates upper bounds 
on the perturbation $\Delta_i$, conditionally on the sign of $S_i$.
By absolute continuity,
almost surely the \(S_{i} X_{i}^{t+1} + \Delta_{i} \neq S_{k}
X_{k}^{t+1} + \Delta_{k}\) for any unseen coordinate \(i \in I_{c}^{t}\)
and any other possibly seen coordinate \(k \neq i\).
In particular, $\Delta_i=D_{i,\pm}$ with probability zero
for \(i \in T_{c}^{t+1}\),
and thus the distribution
on $(\Delta_{i})_{i\in I_c^{t+1}}$ (conditionally on the transcript),
which is the one described above, is absolutely continuous. 
Moreover,
\[ \probability[\Pi_{\leq t}]{I^{t} \nsubseteq \bigcup_{s\leq t} I^s
    \text{ and } \abs*{I^{t+1}} > 1}
  = 0, \]
proving the inductive step.

Finally, by the union bound
\(\probability{E} = 1\) follows.
\end{proof}
\end{lemma}

With the Lemma on unpredictability the proof becomes 
straightforward.
We claim that on event $E$, the oracle answer provided by 
$\mathcal{O}$ can be essentially emulated by the answer provided by
$\tilde{\mathcal{O}}$ on the same point; thus, the trajectory of $A$ 
is equivalent to the trajectory of some algorithm querying
$\tilde{\mathcal{O}}$.
Actually, the oracle \(\mathcal{O}\) emulated
will not be exactly the maximal oracle
on event \(E\) but provide the same \emph{new} information:
one also needs the oracle answers from preceding
iterations to recover the maximal oracle answer.

To prove our claim,
recall that the answer of $\tilde{\mathcal{O}}$
is essentially $(j; s_j,\delta_j)$,
where $j$ is a maximizer coordinate.
We define the emulated oracle \(\mathcal{O}\) to return
$\{f_{r,\gamma} \colon f_{r,\gamma}(x) \leq s_{j} x_{j} + \delta_{j},\quad
r_j = s_j,\,\, \gamma_j=\delta_j\}$.
for query \(x\) on the family of perturbed instances \(\mathcal{F}\).
We observe that on event $E$ by the lemma,
all maximizing coordinates are contained in
\(\bigcup_{s\leq t} I^{s}\),
and thus the maximizing oracle answer in iteration \(t\) is
\[\left\{
    \begin{aligned}
      f_{r,\gamma} \colon
      &
      f_{r, \gamma}(x) \leq s_{j} x_{j}^{t+1} + \delta_{j}
      r_i = s_i,\,\, \gamma_i=\delta_i
      \\
      &
      \text{ for } i \in \bigcup_{s\leq t} I^{s} \text{ and }
      s_{i} x_{i}^{t} + \delta_{i} = s_{j} x_{j}^{t} + \delta_{j}
    \end{aligned}
  \right\}
\]
indeed computable from preceding oracle answers.

By the claim we conclude that
for all $\delta$ under the almost sure event \(E\), we have
$ \expectation[\Delta = \delta]{T}
\geq \frac{M (1 - \varepsilon / K)}{2}$.
By averaging over $\delta$ we obtain
$ \expectation{T} \geq \frac{M (1 - \varepsilon / K)}{2}$.
By choosing $K>0$ arbitrarily large we obtain the desired lower bound.
\end{proof}

%%%%%%%%%%%%%%%%%%%%%%%%%%%%%%%%%%%%%%%%
%%%%%% BOX PERTURBATIONS %%%%%%%%%%%%%%%%%%%%%
%%%%%%%%%%%%%%%%%%%%%%%%%%%%%%%%%%%%%%%%

\subsection{Complexity for the box}
\label{sec:perturbed-box}

For the box case we will first introduce the family construction,
which turns out to be slightly more involved than the one in 
Section~\ref{sec:oracle-compl-box}. Similarly as in the large-scale
case, we first analyze the perturbed family for a fixed perturbation
under the single-coordinate oracle, and then we prove the Lemma
on unpredictability. With this the rest of the proof is analogous to the 
large-scale case and thus left as an exercise.

\begin{theorem}
  \label{th:perturbed-box-lower-bound}
  Let \(L, R > 0\), and
  let \(\mathcal{O}\) be an arbitrary local oracle for the family
  \(\mathcal{F}\) of Lipschitz-continuous convex functions
  on the $L^{\infty}$-ball
  \(B_{\infty}(0,R)\) with Lipschitz constant $L$ in the
  $L^{\infty}$ norm. Then both the distributional
  and the high-probability oracle complexity
  for finding an \(\varepsilon\)-minimum
  is \(  \Omega \left( n \log \frac{LR}{\varepsilon} \right) \).
  
  For bounded-error algorithms with error bound \(P_e\),
  the distributional complexity is
   \(  \Omega \left( (1-P_e)n \log \frac{LR}{\varepsilon}\right) \),
  and the high-probability complexity is
  \(\Omega\left( n \log \frac{LR}{\varepsilon} \right)\).
\end{theorem}

As in Section~\ref{sec:oracle-compl-box}, w.l.o.g. we prove the
Theorem for \(L=R=1\), and recall that w.l.o.g. \(\mathcal{O}\)
is the maximal oracle.

\subsubsection*{One dimensional construction of function family}
First we define the perturbed instances for the one dimensional family.
The multidimensional family will be defined simply as the maximum of one
dimensional functions, as in \eqref{box_family}.

We will utilize different perturbations for each level in the recursive
definition of the functions.
For this reason, in order to preserve convexity, and 
in order to not reveal the behavior of lower levels through perturbations, we 
need to patch the perturbations of consecutive levels in a consistent way. 

Given \(0<\varepsilon\leq 1\), let 
\(M\coloneqq \lfloor\frac{1}{3-\ln\alpha}\ln(1/\varepsilon)\rfloor\) and
$\bar\delta \coloneqq \frac{1-\alpha}{4}(\frac{\alpha}{8})^M$, where
$\alpha \coloneqq 1-8\varepsilon/(5KM)$, and $K$ is a large constant.
Note that for \(K\) large enough \(\alpha>1/e\), independently 
of the values \(\varepsilon\in(0,1]\) and \(M\geq 1\); this way,
we guarantee that 
\(M\geq\lfloor\frac14\ln(1/\varepsilon)\rfloor\). Once we 
have defined our function family we justify our choice for these
parameters.

Let us recall from Section~\ref{sec:oracle-compl-box}  
the recursive definition of intervals \((I_s)_{s\in\{0,1\}^M}\) and properties 
\ref{item:intLength}--\ref{item:bounds}.
We will prove there exists a family 
$\tilde{\mathcal{F}} = \{f_{s,\delta} \colon [-1,1]\to \RR:
\,\,s\in\{0,1\}^l,\, 0<\delta_i\leq \bar\delta, \, i=1,\ldots,M \}$,
satisfying properties \ref{item:intLength}, \ref{item:disjoint-if-not-prefix},
and the analogues of \ref{item:family-increasing}--\ref{item:bounds}
described below.
\begin{enumerate}[label=(G-\arabic*), start=3] % hard-code start number
\item\label{item:perturbed-family-increasing}
  \(f_{s,\delta} \geq f_{\prefix{l}{s},\delta}\)
  with \(f_{s,\delta}(x) = f_{\prefix{l}{s},\delta}(x)\) if and only if \(x \in [-1,1]\setminus 
  \interior{I_{\prefix{l}{s}}^{\delta}}\), where
  \[I_{\prefix{l}{s}}^{\delta}
  \coloneqq I_{\prefix{l}{s}}\left[-1+\left(\frac{2}{\alpha}\right)^l\frac{\delta_{l+1}}{1-\alpha},
  1-\left(\frac{2}{\alpha}\right)^l\frac{\delta_{l+1}}{1-\alpha/2}\right].\]
\item\label{item:perturbed-family-values}
  The function \(f_{s,\delta}\) restricted to the interval \(I_{s}\)
  is of the form
  \begin{equation*}
    f_{s,\delta}(x) = b_{\abs{s},\delta} - \left(\frac{\alpha}{8}\right)^{ \abs{s}}
    + \left(\frac{\alpha}{2}\right)^{ \abs{s}} \abs*{x - I_{s}(0)}
    \qquad x \in I_{s},
  \end{equation*}
 where \(b_{\abs{s},\delta} = f_{s,\delta}(I_{s}(-1)) = 
 f_{s,\delta}(I_{s}(+1))\) is the function value  on the 
 endpoints of \(I_{s}\) (defined inductively on \(|s|\) and
 \(\delta_i\), \(i\leq|s|\)).
  \item\label{item:perturbed-family-bounds}
  For \(t \sqsubseteq s\), we have \(f_{s,\delta}(x) < b_{\abs{t},\delta}\)
  if and only if \(x \in \interior{I_{t}}\).
\end{enumerate}

\begin{figure}[h]
  \centering
  \small
  \begin{tikzpicture}[x=2em,y=3em]
    %% f_{s}
    \draw[name path=f-s] (0,4) coordinate(top left)
    [namex ={\(-1\)}]
    -- ++(4,-4) %($4*(1,-1)$)
    -- ++(4,+4) coordinate(top right) %($4*(1,+1)$)
    [namex = {\(+1\)}]
    ;
    %% old f_{s0}
    \draw[gray, thick] (top left)
    -- ++(2,-2)
    -- ++(1,-1/2)
    -- (top right)
    ;
    %% old f_{s0} shrunk
    \draw[dotted, thick] (top left)
    -- ++(2,-1) coordinate(break)
    [namex = {\(- \frac{1}{2}\)}]
    -- ++(1,-1/4) coordinate(min)
    [namex = {\(- \frac{1}{4}\)}]
    -- coordinate[pos=.3, below=1mm](arrowstart) (top right)
    ;
    %% g_{s0}
   \coordinate [below=.5 of top right](top right shifted);
   \coordinate [below=.5 of top left](top left shifted);
   \draw[dotted, thick, name path=g-s]
    (top left shifted)
    -- ++(2,-1) coordinate(perturbed break)
    -- ++(1,-1/4) coordinate(perturbed min)
    -- coordinate[pos=.3, above=1mm](arrowend) (top right shifted)
    ;
    %% mark constant distance
    \draw[-latex, thick] (arrowstart) -- (arrowend);
    \draw[<->]
    ([xshift=-2mm] top left) -- node[left]{\(\delta_{\size{s}+1}\)}
    ([xshift=-2mm] top left shifted);
    %% perturbed f_s0, compute intersections
    \draw[name intersections={of=f-s and g-s}, thick]
	(top left) --
	(intersection-1) -- (perturbed break) -- (perturbed min)
        -- (intersection-2) -- (top right);
    %% baseline
    \draw[thin] (0,0) node[left]{\(\Delta_{s}\)} -- (8,0);
    %% missing midpoint
    \path (4,0) [namex = {\(0\)}];
    %% level sets
    \draw[dashed] (top left) -- (top right)
    node[right]{\(b_{\abs{s}, \delta}\)};
    \draw[dashed] (perturbed break) -- (top right |- perturbed break)
    node[right]{\(b_{\abs{s} + 1, \delta}\)};
  \end{tikzpicture}
  \caption{Comparison between instance from Section
  \ref{subsec:one_dim_case} (grey line)
  and perturbed one (thick line).}
  \label{fig:perturbed-box}
\end{figure}
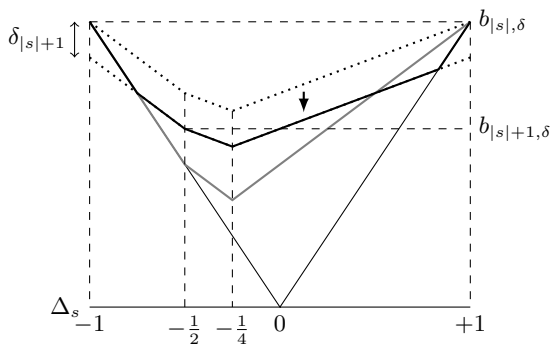

We construct our instance inductively, the case $|s|=0$ being trivial 
($f_{\perp}(x)=|x|$; note this function does not depend on the 
perturbations \(\delta\)). Moreover, let \(b_{0,\delta}=1\), and inductively
\(b_{l+1,\delta}\coloneqq b_{l,\delta}-\frac{\alpha}{2}\left(\frac{\alpha}{8}\right)^l-\delta_{l+1}\).
Suppose now $|s|=l $ and $\delta\in[0,\bar\delta]^M$, and
for simplicity let $s_{l+1}=0$ (the case $s_{l+1}=1$ is analogous). 
By inductive hypothesis 
$f_{s,\delta}(I_{s}(-1)) = f_{s,\delta}(I_{s}(+1))=b_{|s|,\delta}$. 
We consider the perturbed extension given by 
\begin{equation*}
g_{s0,\delta}(x) \coloneqq
b_{l+1,\delta} - \left(\frac{\alpha}{8}\right)^{l + 1} + 
    \left(\frac{\alpha}{2}\right)^{l+1}
    \abs*{x - I_{s} \left(- \frac{1}{4} \right)}
\end{equation*}
if \(x \in I_{s} \left[- \frac{1}{2}, 1\right]\), and
\begin{equation*}
g_{s0,\delta}(x) \coloneqq
 b_{l,\delta}+\alpha\left[f_{s,\delta}(x)-b_{l,\delta}\right] -\delta_{l+1}
\end{equation*}
otherwise.
%\begin{align*}
%  g_{s0,\delta}(x) &\coloneqq
%  \begin{dcases}
%    b_{l+1,\delta} - \left(\frac{\alpha}{8}\right)^{l + 1} + 
%    \left(\frac{\alpha}{2}\right)^{l+1}
%    \abs*{x - I_{s} \left(- \frac{1}{4} \right)}, &
%    \text{if } x \in I_{s} \left[
%      - \frac{1}{2}, 1
%    \right]\\
%    b_{l,\delta}+\alpha\left[f_{s,\delta}(x)-b_{l,\delta}\right]
%    -\delta_{l+1}, 
%    & \text{otherwise.}
%  \end{dcases}
%\end{align*}

We define the new perturbed instance as follows
\[ f_{s0,\delta}(x) = \max\{ g_{s0,\delta}(x), f_{s,\delta}(x)\} \qquad x \in [-1,1]. \]
Note for example that at \(x=I_{s}(-1/2)\) the function \(g_{s0,\delta}\) is 
continuous, and moreover \(g_{s0,\delta}(x)=b_{l+1,\delta}>
b_{l,\delta}-\frac12\left(\frac{\alpha}{8}\right)^l=f_{s,\delta}(x)\), 
where the strict inequality holds by definition of \(\bar\delta\);
similarly, for \(x=I_s(0)\), \(g_{s0,\delta}(x)=b_{l+1,\delta}>f_{s,\delta}(x)\). 
This way, we guarantee that at the interval \(I_{s0}\) the maximum
defining \(f_{s0,\delta}\) is only achieved by \(g_{s0,\delta}\).

The key property of the perturbed instances is the following: 
Since $\delta_{l+1}>0$ then $f_{s0,\delta}$ is smooth at
$I_{s}(-1)$ and $I_s(+1)$, and its local behavior does not depend on
$\delta_{l+1},\ldots,\delta_M$. Furthermore, for all \(x \in [-1,1]\setminus 
\interior{I_{s}^{\delta}} \), 
we have $f_{s0,\delta}(x) =  f_{s,\delta}(x)$, from which is easy to prove 
\ref{item:perturbed-family-increasing}.

Finally, observe that properties \ref{item:intLength}, \ref{item:disjoint-if-not-prefix}, 
\ref{item:perturbed-family-values} and \ref{item:perturbed-family-bounds}
are straightforward to verify. This proves the existence of our family. Moreover, by 
construction, the function defined above is convex, continuous, and has Lipschitz 
constant bounded by \(1\).

To finish our discussion, let us explain the role of these perturbations, and the 
choice of parameters. First observe that the definition of \(g_{s,\delta}\) is
obtained by applying two operations to the extension used in 
Section~\ref{sec:oracle-compl-box}: first we reduce the slope of the extension
by a factor \(\alpha\), and then we `push-down' the function values by an 
additive perturbation \(\delta_{\abs{s}+1}\) (see Figure~\ref{fig:perturbed-box}). 
The motivation for the perturbed 
family is to provide instances with similar structure than in 
Section~\ref{sec:oracle-compl-box}; in particular, we preserve the nesting
property of level sets. The main difference with the perturbed instance is the 
smoothness at \(I_{s}(-1)\), \(I_s(+1)\): by doing this we hide the behavior
(in particular the perturbations) of deeper level sets from its behavior outside
the interior of this level set, for any local oracle. In the multidimensional 
construction the perturbations will have a similar role than in the large-scale 
case, making the maximizer term unique with probability 1 for any oracle query, as 
perturbations in different coordinates will be conditionally independent. This 
process will continue throughout iterations, and the independence of 
perturbations for deeper level sets is crucial for this to happen.

\subsubsection*{Multidimensional construction of the family}
As in the unperturbed case, the obvious multidimensional extension is to
consider the maximum among all coordinates of the one dimensional instance,
namely, for a concatenation of \((nM)\)-dimensional strings \(\{s_i:i\in[n]\}\), \(s\), 
and concatenation of \((nM)\)-dimensional vectors \(\{\delta_i:i\in[n]\}\), \(\delta\), 
let
\begin{equation} 
\label{max-multidim}
f_{s,\delta}(x) \coloneqq \max_{i\in[n]} f_{s_i,\delta_i}(x).
\end{equation}

\paragraph*{Lower bound for fixed perturbation under oracle $\tilde{\mathcal{O}}$}
Note that from \ref{item:intLength} and \ref{item:perturbed-family-bounds} 
the packing property is satisfied when 
\(M=\lfloor\frac{1}{3-\ln\alpha}\ln(1/\varepsilon)\rfloor\). Next, 
emulation by the String Guessing Problem comes from analogous 
results to Lemmas~\ref{lem:answer-nonsmooth-1-dim} and 
\ref{lem:learn-nonsmooth-adversarial}, considering the obvious 
modifications due to the perturbations, and whose proofs are 
thus omitted. This
establishes the lower bound \(\Omega(n\log(1/\varepsilon))\).

\paragraph*{Lower Bounds for \(\mathcal{F}\) under oracle $\mathcal{O}$}
Similarly as in the large-scale case, the fundamental task is to 
prove that w.p. 1 at every iteration the information provided by  
\(\mathcal{O}\) can be emulated by the single-coordinate oracle
\(\tilde{\mathcal{O}}\) studied earlier.

For this, we will analyze the oracle answer, showing that for any nontrivial query
the maximizer in \eqref{max-multidim} is unique w.p. 1. The role of perturbations
is crucial for this analysis. With this in hand, the lower bound comes from an 
averaging argument analogous the large-scale case.

\begin{lemma}[of unpredictability, box case]
\label{lem:unpredictability-box-case}
Let $\prec$ be the lexicographic order defined in \eqref{lexic-order}. 
For an \(\mathcal{O}\)-based algorithm solving family \(\mathcal{F}\) with queries
\(X^1,\ldots,X^T\) let the set of maximizer coordinates be
\[ J^t\coloneqq \{(i,l):\,\, f_{S,\Delta}(X^t)=f_{S_i,\Delta_i}(X^t) , \,\,
b_{l+1,\delta}<f_{S,\Delta}(X^t) \leq b_{l,\delta}\} \]
for \(t\leq T\), and \(J^t=\emptyset\) otherwise. For a query \(t\leq T\) let the
\(i\)-th depth \(l_i\) be such that \((i,l_i)\) is \(\prec\)-maximal among
elements of \(J^{t-1}\) with first coordinate \(i\). Finally, let 
\(J_c^{t} \coloneqq \{(i,l): (i,l)\succ (i,l_i)\}\).

Then the distribution of \((\Delta_{i,h})_{J_c^{t}}\) conditionally on 
\((\Pi_{< t},Q_t)\) is absolutely continuous. Moreover, after the 
oracle answer \(A_{t}\), w.p. 1 either we only obtain (inexact) lower 
bounds on some of the \(\Delta_{i,h}\), or \(J^t\) is a singleton.
\begin{proof}
For \(t<T\), let the active set be defined as
\[
\mathcal{I}^t\coloneqq \interior{\prod_{i=1}^n I_{\prefix{l_i+1}{s_i}}^{\Delta_i}}.
\]

We prove the Lemma by induction on \(t\). The case \(t=1\) clearly 
satisfies that \((I_{i,l})_{(i,l)\in[n]\times [M]}\) is absolutely 
continuous. Next, after the first oracle call, there are two cases: 
first, if the query lies outside the active set \(\mathcal{I}^1\),
then after the oracle answer all what is learnt are lower 
bounds on the perturbations (this since the instance behaves 
as an absolute value function of the maximizer coordinates); 
by absolute continuity these lower bounds are inexact 
w.p. 1. If the query lies in \(\mathcal{I}^1\) then since the perturbations
are absolutely continuous, and since (for all possible realizations 
of \(S\)) the set where the maximizer is not unique is a smaller 
dimensional manifold, the maximizer in \(f_{S,\Delta}\) is unique 
w.p. 1. In this case all bits preceding this maximizer in the 
\(\prec\)-order are learnt, and potentially some perturbations for 
these bits as well.

Next, let \(t\geq 1\), and suppose the Lemma holds up to query \(t\). 
Then we know that \((\Delta_{i,h})_{J_c^{t}}\) is absolutely 
continuous, conditionally on \((\Pi_{< t},Q_{t})\), and that the 
oracle answer \(A_{t}\) is such that w.p. 1 either we only obtain 
(inexact) lower bounds on some \(\Delta_{i,h}\), or \(J^{t}\) 
is a singleton. In the first case, \((\Delta_{i,h})_{J_c^{t}}\) 
remains absolutely continuous (since lower bounds are inexact), 
so clearly the statement holds true for \(t+1\). In the case \(J^{t}\) 
is a singleton, note that \((\Delta_{i,h})_{(i,h)\in J_c^{t+1}}\) remains 
independent and uniform by construction of the function family. 
This way, by performing the same analysis as in the base case 
over the set \(\prod_{i=1}^n I_{\prefix{l_i+1}{s_i}}\) 
we conclude that the Lemma holds for \(t+1\).
\end{proof}
\end{lemma}

Let us define the set
\[ E\coloneqq \bigcap_{t\leq T}
   \{(\Delta_{i,h})_{J_c^{t}} \mbox{ is absolutely continuous } 
  \vee |J^t| \leq 1 \}.\]
By the previous Lemma, \(\probability{E}=1\). It is clear 
that on event \(E\), oracle \(\mathcal{O}\) can be emulated by
\(\tilde{\mathcal{O}}\) by following an analogous approach as in
Section~\ref{sec:ellpPerturb}.
It is left as exercise to derive from this the lower complexity 
bound \(\Omega(n\log(1/\varepsilon))\), and its variants for 
expectation, high probability, and bounded error algorithms. 

\subsubsection{The low-scale case: reduction to the box 
when \(1\leq p<2\)}

Finally, as a consequence of our strong lower bounds for arbitrary 
oracles on the box we derive optimal lower complexity bounds for 
low-scale optimization over \(L^p\) balls for \(1\leq p<2\)

\begin{proposition}
  \label{reduction_ball_to_box:p_leq_2}
  Let $1\leq p<2$,
  and \(\varepsilon \leq n^{-\frac{1}{2} - \delta}\) with \(\delta > 0\).
  There exists a family \(\mathcal{F}\)
  of convex Lipschitz-continuous functions
  in the \(L^{p}\) norm
  with Lipschitz constant \(1\)
  on the \(n\)-dimensional unit Euclidean ball \(B_{p}(0,1)\)
  such that for any local oracle for family \(\mathcal{F}\),
  both the distributional
  and high-probability oracle complexity
  of finding an \(\varepsilon\)-minimum
  under the uniform distribution
  is
  \(\Omega\left( n \log \frac{1}{\varepsilon} \right)\).

 For algorithms with error probability at most \(P_e\),
 the distributional complexity is
 \(  \Omega \left( (1-P_e)n \log \frac{1}{\varepsilon} \right) \)
 and the high probability complexity is
 \(\Omega\left( n \log \frac{1}{\varepsilon} \right)\).
\begin{proof}

  This proof is based on convex geometry and it is
  inspired by \cite{Guzman:2013}.
  
  Let \(\varepsilon \leq 1/ n^{1/2+\delta}\) and
  \(X\coloneqq B_p(0,1)\). By Dvoretzky's Theorem on
  the \(L^p\)-ball \cite[Theorem 4.15]{Pisier:1989}, there exists
  a universal constant \(\alpha\in(0,1)\) (i.e., independent of \(p\)
  and \(n\)), such that for \(k= \lfloor \alpha n\rfloor\) there 
  exists  a subspace \(L\subseteq \RR^n\) of dimension \(k\), 
  and a centered ellipsoid \(E\subseteq L\) such that
  \begin{equation} \label{Dvoretzky}
  \frac{1}{2}E\subseteq X\cap L \subseteq E.
  \end{equation}
  Let \(\{\gamma_i(\cdot):\, i=1,\ldots, k\}\) be linear forms
  on \(L\) such that \(E=\{y\in L:\, \sum_{i=1}^k \gamma_i^2(y) \leq 1\}\).
  By the second inclusion above,
  for every \(i\in[k]\)
  the maximum of \(\gamma_i\) over \(X\cap L\) does not exceed 1,
  whence, by the Hahn-Banach Theorem, the linear form 
  \(\gamma_i(\cdot)\) can be extended from \(L\) to \(\RR^n\) with 
  its maximum over \(X\) not exceeding 1. In  other words, there
  exist \(k\) vectors \(g_i\in\RR^n\), \(1\leq i\leq k\), such that 
  \(\gamma_i(y)=\sprod{g_i}{y}\) for every \(y \in L\) and
  \(\norm[\frac{p}{p-1}]{g_i} \leq 1\), for all \(1 \leq i \leq k\).
Now consider
  the linear mapping 
  \[ x \mapsto Gx \coloneqq
  (\sprod{g_1}{x}, \dotsc, \sprod{g_k}{x}) \colon
  \RR^n\to \RR^k. \]
The operator norm of this mapping induced by the norms
\(\norm[p]{\cdot}\) on the domain and
\(\norm[\infty]{\cdot}\) on the
codomain does not exceed \(1\).
Therefore, for any Lipschitz-continuous function
\(f \colon \RR^k \to \RR\)
with Lipschitz constant \(1\)
in the \(\norm[\infty]{\cdot}\) norm, the function
\(\tilde{f} \colon \RR^n\to \RR\) defined by
\(\tilde{f}(x) = f(G x)\) is
Lipschitz-continuous with constant \(1\) in the
\(L^p\) norm.
We claim (postponing its proof) that the complexity of Lipschitz-continuous functions
in the \(L^p\) norm on \(X \subseteq \RR^n\) is
  lower bounded by the complexity of Lipschitz-continuous functions
  in the \(L^{\infty}\) norm on \( B_{\infty}\left(0,\frac{1}{2\sqrt k}\right) \subseteq \RR^k\)
  (as
  \(G\left(B_{\infty}\left(0,\frac{1}{2\sqrt k}\right)\right) \subseteq \frac12 E\subseteq X\)).
  We conclude that the distributional and high probability oracle complexity
  of the former family is lower bounded by
  \begin{equation*}
  \Omega\left(k \log \dfrac{1}{2\sqrt k \varepsilon} \right) 
    = \Omega\left(n \log \dfrac{1}{\varepsilon \sqrt{n}} \right)
    = \Omega\left(n\log \frac{1}{\varepsilon} \right),
  \end{equation*}
  for large \(n\), since for \(\varepsilon\leq n^{-1/2-\delta}\) with \(\delta>0\)
  we have \( \varepsilon \sqrt n \leq \varepsilon^{\frac{\delta}{1/2+\delta}}\).

  We finish the proof by proving the claim: let \(\mathcal{G}\) be the 
  subfamily of Lipschitz-continuous functions with constant \(1\) for the
  \(L^{\infty}\) norm given by (\ref{max-multidim}), defined 
  on the box \(B_{\infty}(0,1/(2\sqrt k))\) of \(\RR^k\), and let 
  \(\mathcal{F}\) be the respective family of `lifted' instances 
  \(\tilde f \colon \RR^n \to \RR\),
  which are Lipschitz-continuous functions with constant \(1\) for the
  \(L^p\) norm, defined on the unit ball \(B_p(0,1)\) of \(\RR^n\).

  Observe that the maximal oracle \(\mathcal{O}\) on 
  \(\mathcal{G}\) induces the maximal oracle for family 
  \(\mathcal{F}\). Namely, if we let \(\tilde{\mathcal{O}}\) be the oracle
  for family \(\mathcal{F}\) defined by
  \(\tilde{\mathcal{O}}_{\tilde f}(x)=\tilde{\mathcal{O}}_{\tilde g}(x)\) 
  if and only if  \(\mathcal{O}_{f}(Gx)=\mathcal{O}_{g}(Gx)\), then
  it is easy to see that \(\tilde{\mathcal{O}}\) is the maximal oracle
  for \(\mathcal{F}\). This way, any oracle for \(\mathcal{F}\) can be 
  emulated by an oracle for \(\mathcal{G}\), and thus by
  Lemma~\ref{lem:emulator-complexity} lower 
  bounds for \(\mathcal{G}\) also hold for \(\mathcal{F}\).

\end{proof}
\end{proposition}

%\section*{Acknowledgements}
%\label{sec:acknowledgements} 
%The authors would like to thank Arkadi Nemirovski and Fran\c{c}ois Glineur for the valuable
%discussions. Research reported in this paper was partially supported
%by NSF grants  CMMI-1232623 and CMMI-1300144.

\bibliographystyle{IEEEtran}
\bibliography{IEEEabrv,Bibliography}

\begin{thebibliography}{10}
\providecommand{\url}[1]{#1}
\csname url@samestyle\endcsname
\providecommand{\newblock}{\relax}
\providecommand{\bibinfo}[2]{#2}
\providecommand{\BIBentrySTDinterwordspacing}{\spaceskip=0pt\relax}
\providecommand{\BIBentryALTinterwordstretchfactor}{4}
\providecommand{\BIBentryALTinterwordspacing}{\spaceskip=\fontdimen2\font plus
\BIBentryALTinterwordstretchfactor\fontdimen3\font minus
  \fontdimen4\font\relax}
\providecommand{\BIBforeignlanguage}[2]{{%
\expandafter\ifx\csname l@#1\endcsname\relax
\typeout{** WARNING: IEEEtran.bst: No hyphenation pattern has been}%
\typeout{** loaded for the language `#1'. Using the pattern for}%
\typeout{** the default language instead.}%
\else
\language=\csname l@#1\endcsname
\fi
#2}}
\providecommand{\BIBdecl}{\relax}
\BIBdecl

\bibitem{BenTal:2001}
A.~Ben-Tal, T.~Margalit, and A.~Nemirovski, ``The ordered subsets mirror
  descent optimization method with applications to tomography,'' \emph{SIAM J.
  Optim}, vol.~12, p. 2001, 2001.

\bibitem{Jaggi:2013}
S.~Lacoste-Julien, M.~Jaggi, M.~Schmidt, and P.~Pletscher, ``Block-coordinate
  {Frank-Wolfe} optimization for structural {SVMs},'' in \emph{Proceedings of
  the 30th International Conference on Machine Learning (ICML-13)}, 2013, pp.
  53--61.

\bibitem{TV_Beck:2009}
A.~Beck and M.~Teboulle, ``Fast gradient-based algorithms for constrained total
  variation image denoising and deblurring problems,'' \emph{IEEE Transactions
  on Image Processing}, vol.~18, no.~11, pp. 2419--2434, 2009.

\bibitem{Wright:2010}
M.~Zhu, S.~J. Wright, and F.~T. Chan, ``Duality-based algorithms for
  total-variation-regularized image restoration,'' \emph{Comput. Optim. Appl.},
  vol.~47, no.~3, pp. 377--400, Nov. 2010.

\bibitem{ell1_Beck:2009}
A.~Beck and M.~Teboulle, ``A fast iterative shrinkage-thresholding algorithm
  for linear inverse problems,'' \emph{SIAM J. Img. Sci.}, vol.~2, no.~1, pp.
  183--202, Mar. 2009.

\bibitem{Nemirovski:2013}
Y.~Nesterov and A.~Nemirovski, ``On first-order algorithms for $\ell_1$/nuclear
  norm minimization,'' \emph{Acta Numerica}, vol.~22, pp. 509--575, 4 2013.

\bibitem{Nemirovski:1983}
A.~Nemirovski and D.~Yudin, \emph{{Problem complexity and method efficiency in
  optimization}}, 1st~ed.\hskip 1em plus 0.5em minus 0.4em\relax Wiley
  -Interscience, 1983.

\bibitem{Nemirovski:1994}
\BIBentryALTinterwordspacing
A.~Nemirovski, ``Efficient methods in convex programming,'' 1994, lecture
  notes. [Online]. Available:
  \url{http://www2.isye.gatech.edu/~nemirovs/Lect_EMCO.pdf}
\BIBentrySTDinterwordspacing

\bibitem{Nesterov:2004}
Y.~Nesterov, \emph{{Introductory lectures on convex optimization: a basic
  course}}, 1st~ed.\hskip 1em plus 0.5em minus 0.4em\relax Springer
  Netherlands, 2004.

\bibitem{braverman2011information}
M.~Braverman and A.~Rao, ``Information equals amortized communication,'' in
  \emph{Foundations of Computer Science (FOCS), 2011 IEEE 52nd Annual Symposium
  on}.\hskip 1em plus 0.5em minus 0.4em\relax IEEE, 2011, pp. 748--757.

\bibitem{braverman2012information3}
M.~Braverman, A.~Garg, D.~Pankratov, and O.~Weinstein, ``From information to
  exact communication,'' in \emph{Electronic Colloquium on Computational
  Complexity (ECCC)}, vol.~19, 2012, p. 171.

\bibitem{p2011unifying}
M.~P\v{a}tra\c{s}cu, ``{Unifying the landscape of cell-probe lower bounds},''
  \emph{SIAM Journal on Computing}, vol.~40, no.~3, pp. 827--847, 2011.

\bibitem{dasgupta2012sparse}
A.~Dasgupta, R.~Kumar, and D.~Sivakumar, ``Sparse and lopsided set disjointness
  via information theory,'' in \emph{Approximation, Randomization, and
  Combinatorial Optimization. Algorithms and Techniques}.\hskip 1em plus 0.5em
  minus 0.4em\relax Springer, 2012, pp. 517--528.

\bibitem{braverman2012information}
\BIBentryALTinterwordspacing
M.~Braverman and A.~Moitra, ``An information complexity approach to extended
  formulations,'' in \emph{Proceedings of {STOC}}, Jun. 2013, pp. 161--170.
  [Online]. Available: \url{https://eccc.weizmann.ac.il/report/2012/131/}
\BIBentrySTDinterwordspacing

\bibitem{BP2013commInfo}
G.~Braun and S.~Pokutta, ``{Common information and unique disjointness},''
  \emph{{Proceedings of FOCS}}, 2013.

\bibitem{BFP2013}
G.~Braun, S.~Fiorini, and S.~Pokutta, ``Average case polyhedral complexity of
  the maximum stable set problem,'' \emph{Mathematical Programming}, vol. 160,
  pp. 407--431, Mar. 2016.

\bibitem{chakrabarti2013information}
A.~Chakrabarti, G.~Cormode, R.~Kondapally, and A.~McGregor, ``Information cost
  tradeoffs for augmented index and streaming language recognition,''
  \emph{SIAM Journal on Computing}, vol.~42, no.~1, pp. 61--83, 2013.

\bibitem{goel2009approximability}
G.~Goel, C.~Karande, P.~Tripathi, and L.~Wang, ``Approximability of
  combinatorial problems with multi-agent submodular cost functions,'' in
  \emph{Foundations of Computer Science, 2009. FOCS'09. 50th Annual IEEE
  Symposium on}.\hskip 1em plus 0.5em minus 0.4em\relax IEEE, 2009, pp.
  755--764.

\bibitem{iwata2009submodular}
S.~Iwata and K.~Nagano, ``Submodular function minimization under covering
  constraints,'' in \emph{Foundations of Computer Science, 2009. FOCS'09. 50th
  Annual IEEE Symposium on}.\hskip 1em plus 0.5em minus 0.4em\relax IEEE, 2009,
  pp. 671--680.

\bibitem{svitkina2011submodular}
Z.~Svitkina and L.~Fleischer, ``Submodular approximation: Sampling-based
  algorithms and lower bounds,'' \emph{SIAM Journal on Computing}, vol.~40,
  no.~6, pp. 1715--1737, 2011.

\bibitem{chudak2007efficient}
F.~A. Chudak and K.~Nagano, ``Efficient solutions to relaxations of
  combinatorial problems with submodular penalties via the lov{\'a}sz extension
  and non-smooth convex optimization,'' in \emph{Proceedings of the eighteenth
  annual ACM-SIAM symposium on Discrete algorithms}.\hskip 1em plus 0.5em minus
  0.4em\relax Society for Industrial and Applied Mathematics, 2007, pp. 79--88.

\bibitem{Bockenhauer:2013}
H.-J. B{\"o}ckenhauer, J.~Hromkovi{\v c}, D.~Komm, S.~Krug, J.~Smula, and
  A.~Sprock, ``The string guessing problem as a method to prove lower bounds on
  the advice complexity,'' \emph{Theoretical Computer Science}, vol. 554, pp.
  95--108, Oct.~16 2014.

\bibitem{jaggi2013revisiting}
M.~Jaggi, ``Revisiting {Frank-Wolfe}: Projection-free sparse convex
  optimization,'' in \emph{Proceedings of the 30th International Conference on
  Machine Learning (ICML-13)}, 2013, pp. 427--435.

\bibitem{lan2013complexity}
\BIBentryALTinterwordspacing
G.~Lan, ``The complexity of large-scale convex programming under a linear
  optimization oracle,'' 2013. [Online]. Available:
  \url{https://bpb-us-w2.wpmucdn.com/sites.gatech.edu/dist/f/330/files/2016/02/OptCndG6-26.pdf}
\BIBentrySTDinterwordspacing

\bibitem{ghPlaying2013}
D.~Garber and E.~Hazan, ``{Playing Non-linear Games with Linear Oracles},''
  \emph{{Proceedings of FOCS}}, 2013.

\bibitem{Raginsky:2011}
M.~Raginsky and A.~Rakhlin, ``Information-based complexity, feedback and
  dynamics in convex programming,'' \emph{IEEE Transactions on Information
  Theory}, vol.~57, no.~10, pp. 7036--7056, 2011.

\bibitem{Agarwal:2012}
A.~Agarwal, P.~L. Bartlett, P.~D. Ravikumar, and M.~J. Wainwright,
  ``Information-theoretic lower bounds on the oracle complexity of stochastic
  convex optimization,'' \emph{IEEE Transactions on Information Theory},
  vol.~58, no.~5, pp. 3235--3249, 2012.

\bibitem{Sridharan}
\BIBentryALTinterwordspacing
N.~Srebro and K.~Sridharan, ``On convex optimization, fat shattering and
  learning,'' Unpublished, 2012. [Online]. Available:
  \url{http://ttic.uchicago.edu/~karthik/optfat.pdf}
\BIBentrySTDinterwordspacing

\bibitem{GuzmanPhD}
C.~Guzm{\'a}n, ``Information, complexity and structure in convex
  optimization,'' Ph.D. dissertation, Georgia Institute of Technology, May
  2015.

\bibitem{AroraBarakBook}
S.~Arora and B.~Barak, \emph{Computational complexity}.\hskip 1em plus 0.5em
  minus 0.4em\relax Cambridge: Cambridge University Press, 2009.

\bibitem{cover2006elements}
T.~Cover and J.~Thomas, \emph{Elements of information theory}.\hskip 1em plus
  0.5em minus 0.4em\relax Wiley-interscience, 2006.

\bibitem{Guzman:2013}
C.~Guzm{\'a}n and A.~Nemirovski, ``On lower complexity bounds for large-scale
  smooth convex optimization,'' \emph{Journal of Complexity}, vol.~31, no.~1,
  pp. 1--14, Feb. 2015.

\bibitem{Pisier:1989}
G.~Pisier, \emph{{The Volume of Convex Bodies and Banach Space Geometry}},
  1st~ed.\hskip 1em plus 0.5em minus 0.4em\relax Cambridge University Press,
  1989.

\end{thebibliography}

\begin{IEEEbiography}[{\includegraphics[width=1in,height=1.25in,
    clip,keepaspectratio]{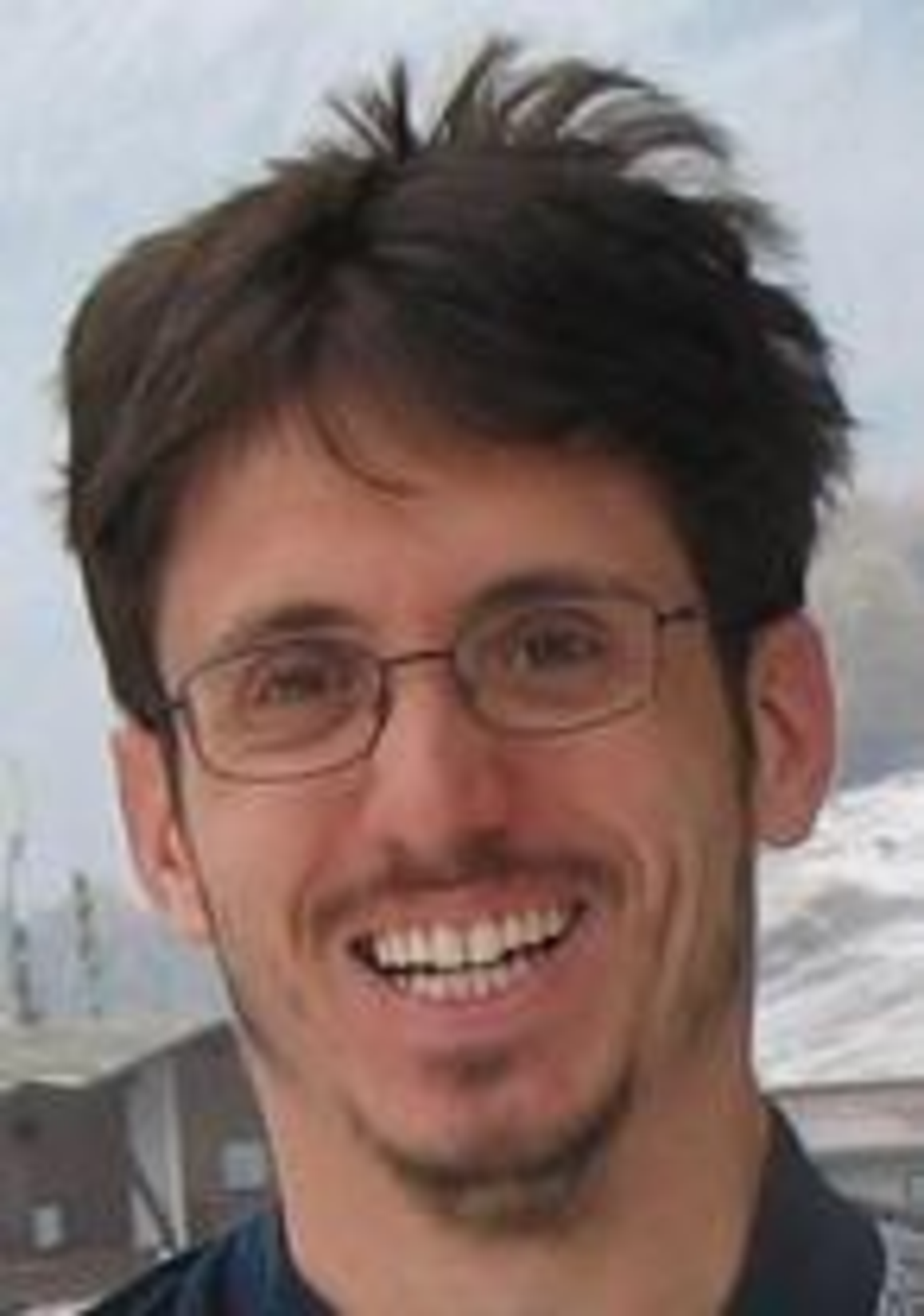}}]{Gábor Braun}
  % current position
  is currently a Postdoctoral Fellow at the Georgia Institute of
  Technology working on complexity of combinatorial
  problems under various programming paradigms like LPs and SDPs,
  and teaching engineering optimization.
  % prior significant professional experience
  He has been a Research Fellow at University Duisburg-Essen, Germany
  working on relationship between logic
  and module theory from algebra,
  and also did research at E\"otv\"os Loránd University, Hungary
  and Alfr\'ed R\'enyi Institute, Hungary
  in group theory and isolated surface singularities.
  He was a teaching assistant at University Leipzig, Germany
  teaching courses in algebraic automata theory and calculus.
  % technical interests
  His research is focused on combining algebra with other fields,
  including topology, logic, information theory and complexity theory.
  % education
  He has received a Master's degree in Mathematics at E\"otv\"os Lor\'and
  University, Budapest, Hungary, and a Ph.D. in Mathematics at the
  University Duisburg-Essen in Germany.
  % awards and important activities
  % professional affiliations
\end{IEEEbiography}

\begin{IEEEbiography}[{\includegraphics[width=1in,height=1.25in,clip,
    keepaspectratio]{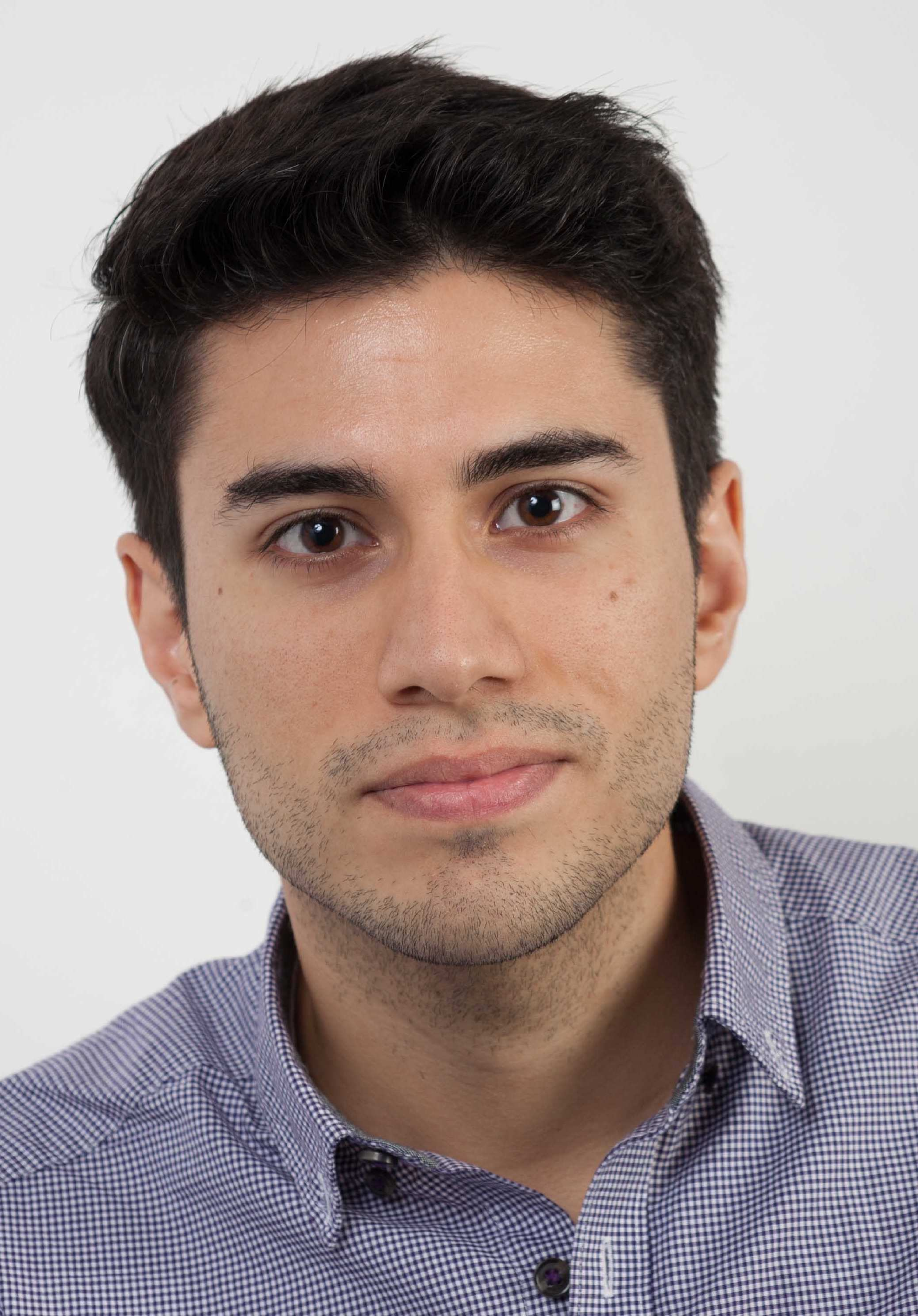}}]{Crist\'obal Guzm\'an} 
    is an Assistant Professor in Mathematical and Computational 
Engineering at Pontificia Universidad Cat\'olica de Chile. He worked
as a Postdoctoral Fellow in the Networks \& Optimization group at Centrum Wiskunde \& Informatica, Amsterdam. After joining Pontificia Universidad Cat\'olica de Chile as Assistant Professor in 2016, he was a Visiting Professor at Centrum Wiskunde \& Informatica.
%technical interests
His research interests include large-scale convex optimization, network 
equilibrium, and machine learning. 
%education
Crist\'obal Guzm\'an received a Mathematical Engineering degree 
from Universidad de Chile, Santiago, Chile in 2010, and a Ph.D. in 
Algorithms, Combinatorics \& Optimization  from Georgia Institute of Technology, Georgia, USA in 2015. 
%awards and important activities
He received a FONDECYT Iniciaci\'on grant in 2017.
\end{IEEEbiography}

\begin{IEEEbiography}[{\includegraphics[width=1in,height=1.25in,clip,
    keepaspectratio]{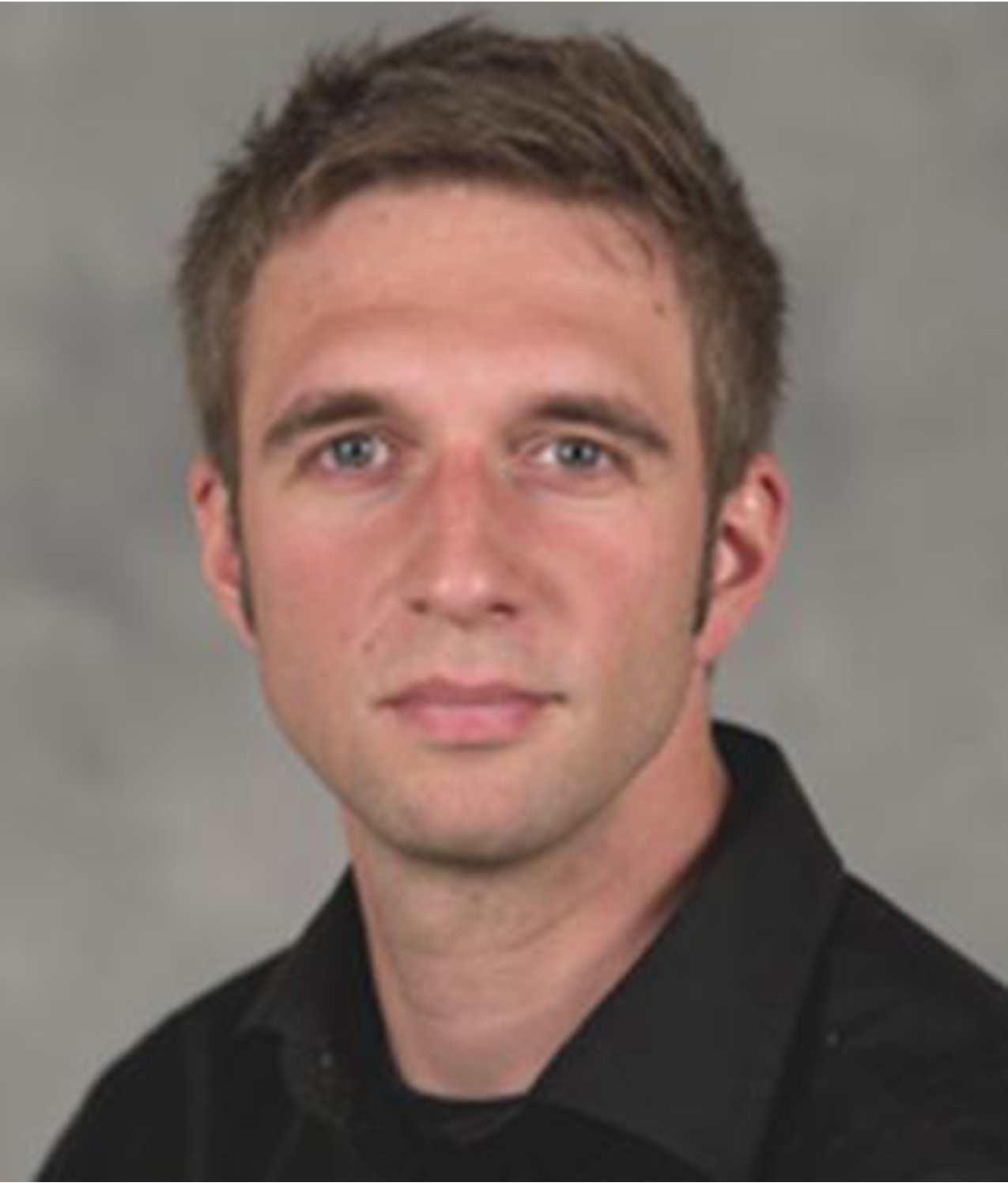}}]{Sebastian Pokutta}
  %Sebastian Pokutta 
  is the David M. McKenney Family Associate Professor
  in the School of Industrial and Systems Engineering at Georgia
  Tech. He worked as a Postdoctoral Fellow at the MIT Operations Research
  Center where the topic of his research was combinatorial
  optimization and cutting-plane procedures. Upon completion of his
  postdoctoral fellowship at MIT Sebastian Pokutta was appointed as an
  optimization specialist at ILOG where he worked on production
  planning and supply chain optimization within the steel industry,
  automotive industry, and energy industry. In early 2008, he joined
  KDB Krall Demmel Baumgarten a state-of-the art risk management
  practice and developed risk management methodologies at top tier
  banks.  Prior to joining the Friedrich-Alexander-University of
  Erlangen-N\"urnberg as an Assistant Professor in 2011 he was a
  visiting lecturer at MIT.
  % technical interests
  Sebastian Pokutta's research concentrates on combinatorial
  optimization, machine learning, and information theory, and
  in particular focuses on cutting-plane methods, extended
  formulations, and the combination of discrete optimization and machine learning. His applied work focuses on deploying these methods within the broader field of engineering and finance. 
  % education
  He received both his master's degree and Ph.D. in mathematics from
  the University of Duisburg-Essen in Germany and he received the NSF CAREER Award in 2015.
  % awards and important activities
  % professional affiliations
\end{IEEEbiography}

\end{document}